\documentclass{amsart}
\usepackage{amssymb}
\usepackage[nobysame]{amsrefs}
\usepackage{mathpazo}
\usepackage{mathrsfs}
\usepackage[usesnames,svgnames]{xcolor}
\usepackage{todonotes}
\usepackage{tikz,pgf}
\usepackage[colorlinks=true,
	urlcolor=MidnightBlue,
	citecolor=DarkGreen]{hyperref}
\pgfdeclarelayer{background}
\pgfsetlayers{background,main}
\newtheorem{theorem}{Theorem}[section]
\newtheorem{proposition}[theorem]{Proposition}
\newtheorem{lemma}[theorem]{Lemma}
\newtheorem{corollary}[theorem]{Corollary}
\newtheorem{example}[theorem]{Example}

\newcommand{\defn}[1]{{\color{DarkGreen}\emph{#1}}}
\newcommand{\ie}{\text{i.e.}\;}
\newcommand{\nc}{N\!C}
\newcommand{\pe}{P\!E}
\renewcommand{\qq}{\mathbf{q}}
\renewcommand{\aa}{\mathbf{a}}
\newcommand{\xx}{\mathbf{x}}
\newcommand{\yy}{\mathbf{y}}
\newcommand{\zz}{\mathbf{z}}
\newcommand{\ww}{\mathbf{w}}
\newcommand{\zero}{\mathbf{0}}
\newcommand{\one}{\mathbf{1}}
\renewcommand{\AA}{\mathcal{A}}
\newcommand{\BB}{\mathcal{B}}
\newcommand{\CC}{\mathscr{C}}
\newcommand{\DD}{\mathscr{D}}
\newcommand{\HH}{\mathscr{H}}
\newcommand{\LL}{\mathcal{L}}
\newcommand{\RR}{\mathcal{R}}
\newcommand{\SM}{\mathcal{S}}
\renewcommand{\SS}{\mathfrak{S}}
\newcommand{\rk}{\text{rk}}
\newcommand{\dref}{\leq_{\text{{\normalfont dref}}}}
\newcommand{\pchn}{\leq_{\text{{\normalfont pchn}}}}
\newcommand{\cat}{\text{Cat}}
\author{Henri M{\"u}hle}
\address{Institut f{\"u}r Algebra, Technische Universit{\"a}t Dresden, 01069 Dresden, Germany.}
\email{henri.muehle@tu-dresden.de}
\title{Two Posets of Noncrossing Partitions Coming From Undesired Parking Spaces}

\begin{document}

\allowdisplaybreaks

\begin{abstract}
	Consider the noncrossing set partitions of an $n$-element set which either do not contain the block $\{n-1,n\}$, or which do not contain the singleton block $\{n\}$ whenever $1$ and $n-1$ are in the same block.  In this article we study the subposet of the noncrossing partition lattice induced by these elements, and show that it is a supersolvable lattice, and therefore lexicographically shellable.  We give a combinatorial model for the NBB bases of this lattice and derive an explicit formula for the value of its M{\"o}bius function between least and greatest element.  
	
	This work is motivated by a recent article by M.~Bruce, M.~Dougherty, M.~Hlava\-cek, R.~Kudo, and I.~Nicolas, in which they introduce a subposet of the noncrossing partition lattice that is determined by parking functions with certain forbidden entries.  In particular, they conjecture that the resulting poset always has a contractible order complex.  We prove this conjecture by embedding their poset into ours, and showing that it inherits the lexicographic shellability.
\end{abstract}

\maketitle

\section{Introduction}
	\label{sec:introduction}
A set partition of $[n]=\{1,2,\ldots,n\}$ is noncrossing if there are no indices $i<j<k<l$ such that $i,k$ and $j,l$ belong to distinct blocks.  Let us denote the set of all noncrossing set partitions by $\nc_{n}$.  We can partially order noncrossing set partitions by dual refinement, meaning that $\xx\in\nc_{n}$ is smaller than $\yy\in\nc_{n}$ if every block of $\xx$ is contained in a block of $\yy$.  Let us denote this partial order by $\dref$.  

The lattice $(\nc_{n},\dref)$ of noncrossing set partitions is a remarkable poset with a rich combinatorial structure.  It was introduced by G.~Kreweras in the early 1970s~\cite{kreweras72sur}, and has gained a lot of attention since then.  It has, among other things, surprising ties to group theory, algebraic topology, representation theory of the symmetric group, and free probability.  See \cite{simion00noncrossing} and \cite{mccammond06noncrossing} for surveys on these lattices.

A parking function of length $n$ is a function on an $n$-element set with the property that the preimage of $[k]$ has at least $k$ elements for every $k\leq n$.  They were introduced in \cite{konheim66occupancy}, and play an important role in the study of the spaces of diagonal harmonics, see~\cite{haiman94conjectures} and \cite{haglund08catalan}*{Chapter~5}.

The maximal chains of $(\nc_{n},\dref)$ are naturally in bijection with parking functions of length $n-1$, see \cite{stanley97parking}.  This connection was used in \cite{bruce16decomposition} to define a subposet of $(\nc_{n},\dref)$ as follows.  Fix some $k\leq n$ and take the set of all parking functions which do not have $k$ in their image, but every value larger than $k$.  Let us consider the poset $(\pe_{n,k},\pchn)$, which is the subposet of $(\nc_{n},\dref)$ determined by the maximal chains corresponding to these parking functions.  In the case where $n=k$ we simply write $(\pe_{n},\pchn)$.  For $n\leq 2$, the poset $(\pe_{n},\pchn)$ is the empty poset.

Let $\zero$ denote the discrete partition into singleton blocks, and let $\one$ denote the full partition into a single block.  It is the statement of \cite{bruce16decomposition}*{Theorem~C} that the M{\"o}bius function of $(\pe_{n,k},\pchn)$ always vanishes between $\zero$ and $\one$.  It was moreover conjectured there that the order complex of $(\pe_{n,k},\pchn)$ with $\zero$ and $\one$ removed is contractible.  The main purpose of this article is to prove this conjecture.  

In fact we show that $(\pe_{n},\pchn)$ is lexicographically shellable, which together with the aforementioned result on the M{\"o}bius function establishes the following.

\begin{theorem}\label{thm:pf_chain_shellable}
	For $n\geq 3$ the poset $(\pe_{n},\pchn)$ is lexicographically shellable.
\end{theorem}

The following is an immediate corollary of Theorem~\ref{thm:pf_chain_shellable} and \cite{bruce16decomposition}*{Theorem~C}.

\begin{corollary}\label{cor:pf_chain_contractible}
	For $n\geq 3$ the order complex of $(\pe_{n},\pchn)$ with $\zero$ and $\one$ removed is contractible.
\end{corollary}

Theorem~3.5 in \cite{bruce16decomposition} states that $(\pe_{n,k},\pchn)$ is isomorphic to the direct product of $(\pe_{k},\pchn)$ and the Boolean lattice of rank $n-k$.  Since the latter is known to be lexicographically shellable \cite{bjorner80shellable}*{Theorem~3.7}, and lexicographic shellability is preserved under taking direct products \cite{bjorner80shellable}*{Theorem~4.3}, Theorem~\ref{thm:pf_chain_shellable} indeed suffices to resolve the main conjecture of \cite{bruce16decomposition}.  

In order to prove Theorem~\ref{thm:pf_chain_shellable}, we take a detour through a slightly larger subposet of $(\nc_{n},\dref)$.  In fact, we consider the induced subposet $(\pe_{n},\dref)$, and show that it is a supersolvable lattice.  

\begin{theorem}\label{thm:pf_element_supersolvable}
	For $n\geq 3$ the poset $(\pe_{n},\dref)$ is a supersolvable lattice.
\end{theorem}

It is well known that supersolvable lattices possess an edge-labeling that implies their lexicographic shellability~\cite{bjorner80shellable}*{Theorem~3.7}.  The last step in proving Theorem~\ref{thm:pf_chain_shellable} is to show that the restriction of this edge-labeling to $(\pe_{n},\pchn)$ retains its crucial properties.  Observe that for $n\geq 5$, the poset $(\pe_{n},\pchn)$ is not a lattice.

We remark that the edge-labeling coming from Theorem~\ref{thm:pf_element_supersolvable} differs from the usual labeling of $(\nc_{n},\dref)$, which is defined as follows.  If $\xx\lessdot_{\text{\normalfont dref}}\yy$, then there are two blocks $B,B'$ in $\xx$ that are joined in $\yy$.  If the smallest element of $B$ is smaller than the smallest element of $B'$, then the label of this cover relation is $n$ minus the largest element of $B$ that is smaller than every element in $B'$.  The restriction of this labeling to $(\pe_{n},\pchn)$ does, however, not have the properties necessary to guarantee lexicographic shellability.

The last main result of this article is the explicit computation of the value of the M{\"o}bius function in $(\pe_{n},\dref)$ between $\zero$ and $\one$.

\begin{theorem}\label{thm:pf_element_mobius}
	For $n\geq 3$ we have
	\begin{displaymath}
		\mu_{(\pe_{n},\dref)}(\zero,\one)=(-1)^{n-1}\frac{4}{n}\binom{2n-5}{n-4},
	\end{displaymath}
	which is \cite{sloane}*{A099376} up to sign.
\end{theorem}

We prove Theorem~\ref{thm:pf_element_mobius} by using A.~Blass and B.~Sagan's NBB bases~\cite{blass97mobius}.  In fact we give a combinatorial model in terms of trees for these NBB bases, from which we derive their enumeration.

\medskip

The rest of the article is organized as follows.  In Section~\ref{sec:preliminaries} we recall the necessary lattice- and poset-theoretic notions (Section~\ref{sec:posets_lattices}), and formally define noncrossing set partitions (Section~\ref{sec:noncrossing_set_partitions}).  In Section~\ref{sec:pf_element_poset} we define the poset $(\pe_{n},\dref)$, and prove Theorem~\ref{thm:pf_element_supersolvable} (Section~\ref{sec:pf_element_structure}), and Theorem~\ref{thm:pf_element_mobius} (Section~\ref{sec:pf_element_mobius}).  In Section~\ref{sec:pf_chain_poset} we turn our attention to the poset $(\pe_{n},\pchn)$ and conclude the proof of Theorem~\ref{thm:pf_chain_shellable}. 

\section{Preliminaries}
	\label{sec:preliminaries}
\subsection{Posets and Lattices}
	\label{sec:posets_lattices}
Let $\LL=(L,\leq)$ be a finite partially ordered set (\defn{poset} for short).  If $\LL$ has a least and a greatest element (denoted by $\hat{0}$ and $\hat{1}$, respectively), then $\LL$ is \defn{bounded}.  If any two elements $x,y\in L$ have a least upper bound (their \defn{join}; denoted by $x\vee y$) and a greatest lower bound (their \defn{meet}; denoted by $x\wedge y$), then $\LL$ is a \defn{lattice}.  

An element $y\in L$ \defn{covers} another element $x\in L$ if $x<y$ and for all $z\in L$ with $x\leq z\leq y$ we have $x=z$ or $z=y$.  We then write $x\lessdot y$, and we sometimes say that $(x,y)$ is a \defn{cover relation}.  If $\LL$ has a least element $\hat{0}$, then any element covering $\hat{0}$ is an \defn{atom}.

A \defn{chain} is a subset $X\subseteq L$ that can be written as $C=\{x_{1},x_{2},\ldots,x_{k}\}$ such that $x_{1}\leq x_{2}\leq\cdots\leq x_{k}$.  A chain is \defn{saturated} if it can be written as $x_{1}\lessdot x_{2}\lessdot\cdots\lessdot x_{k}$.  A saturated chain is \defn{maximal} if it contains a minimal and a maximal element of $\LL$.  Let $\CC(\LL)$ denote the set of maximal chains of $\LL$.

The \defn{rank} of $\LL$ is one less than the maximum size of a maximal chain; denoted by $\rk(\LL)$.  We say that $\LL$ is \defn{graded} if all maximal chains have the same size.  An \defn{interval} of $\LL$ is a set $[x,y]=\{z\mid x\leq z\leq y\}$.  

Two lattice elements $x,z\in L$ form a \defn{modular pair} if for all $y\leq z$ holds that $(y\vee x)\wedge z=y\vee(x\wedge z)$; we then usually write $xMz$.  Moreover, $x\in L$ is \defn{left-modular} if $xMz$ for all $z\in L$.  If $x$ satisfies both $xMz$ and $zMx$ for all $z\in L$, then $x$ is \defn{modular}.  A maximal chain is \defn{(left-)modular} if it consists entirely of (left-)modular elements.

A lattice is \defn{modular} if all its elements are modular, and it is \defn{left-modular} if it contains a left-modular chain.  A lattice is \defn{supersolvable} if it contains a maximal chain $M$ with the property that for every chain $C$ the sublattice generated by $M$ and $C$ is distributive.  (In other words, the smallest sublattice containing $M$ and $C$ is distributive.)  Chains with this property are called \defn{$M$-chains}.  It follows from \cite{stanley72supersolvable}*{Proposition~2.1} that every element of an $M$-chain is modular, and supersolvable lattices are therefore left-modular.  For graded lattices, these two notions actually coincide.

\begin{theorem}[\cite{mcnamara06poset}*{Theorem~2}]\label{thm:graded_left_modular}
	A finite graded lattice is left-modular if and only if it is supersolvable.
\end{theorem}

For any bounded poset $\LL=(L,\leq)$ let $\HH(\LL)=\bigl\{(x,y)\mid x\lessdot y\bigr\}$ denote the set of cover relations of $\LL$.  An \defn{edge-labeling} of $\LL$ is a map $\lambda:\HH(\LL)\to\Lambda$, for some poset $(\Lambda,\prec)$.  For a saturated chain $C=\{x_{1},x_{2},\ldots,x_{k}\}$ we denote by $\lambda(C)=\bigl(\lambda(x_{1},x_{2}),\lambda(x_{2},x_{3}),\ldots,\lambda(x_{k-1},x_{k})\bigr)$ the associated sequence of edge-labels.  We then say that $C$ is \defn{rising} if $\lambda(C)$ is strictly increasing with respect to $\prec$.  An edge-labeling of $\LL$ is an \defn{EL-labeling} if the following two conditions hold for every interval $[x,y]$ of $\LL$: (i) there exists a unique rising maximal chain $C$ in $[x,y]$, and (ii) for every other maximal chain $C'$ of $[x,y]$ we have that $\lambda(C)$ is lexicographically smaller than $\lambda(C')$.  A poset that admits an EL-labeling is \defn{EL-shellable}.

If $\rk(\LL)=n$, and $\lambda$ is an EL-labeling of $\LL$ such that for every maximal chain $C$ the entries in $\lambda(C)$ are all distinct members of $[n]$, then $\lambda$ is an \defn{$\SS_{n}$ EL-labeling}.

\begin{theorem}[\cite{liu99left}]\label{thm:left_modular_snelling}
	Let $\LL=(L,\leq)$ be a left-modular lattice of length $n$ with left-modular chain $x_{0}\lessdot x_{1}\lessdot\cdots\lessdot x_{n}$.  The labeling 
	\begin{equation}\label{eq:left_modular_labeling}
		\lambda(y,z)=\min\{i\mid y\vee x_{i}\wedge z=z\}
	\end{equation}
	is an $\SS_{n}$ EL-labeling of $\LL$.
\end{theorem}

\begin{theorem}[\cite{mcnamara03labelings}*{Theorem~1}]\label{thm:supersolvable_snellable}
	A finite graded lattice of length $n$ is supersolvable if and only if it is $\SS_{n}$ EL-shellable.
\end{theorem}

The existence of an EL-labeling of $\LL$ has further implications on the homotopy type of the \defn{order complex} associated to $\LL$, \ie the simplicial complex whose faces are the chains of $\LL$.  

\begin{theorem}[\cite{bjorner96shellable}*{Theorem~5.9}]\label{thm:shellable_wedge}
	Let $\LL$ be a bounded graded poset of rank $n$ with $\mu(\hat{0},\hat{1})=k$.  If $\LL$ is EL-shellable, then the order complex of $\LL$ with $\hat{0}$ and $\hat{1}$ removed has the homotopy type of a wedge of $\lvert k\rvert$-many $(n-2)$-dimensional spheres.  Moreover, $k$ is precisely the number of maximal chains of $\LL$ with weakly decreasing label sequence.
\end{theorem}

\subsection{Noncrossing Set Partitions}
	\label{sec:noncrossing_set_partitions}
A \defn{set partition} of $n$ is a covering $\xx=\bigl\{B_{1},B_{2},\ldots,B_{s}\}$ of $[n]$ into non-empty, mutually disjoint sets; which we call \defn{blocks}.  Let $\Pi_{n}$ denote the set of all set partitions of $n$.  For $i,j\in[n]$ and $\xx\in\Pi_{n}$ we write $i\sim_{\xx} j$ if there is $B\in\xx$ with $i,j\in B$.  It is easily seen that $\sim_{\xx}$ is an equivalence relation; in fact set partitions of $[n]$ and equivalence relations on $[n]$ are in bijection.  Let $\zero$ be the discrete partition which consists of $n$ singleton blocks, and let $\one$ be the full partition which consists only of a single block.

A set partition $\xx$ is \defn{noncrossing} if for any four indices $1\leq i<j<k<l\leq n$ the relations $i\sim_{\xx} k$ and $j\sim_{\xx} l$ imply $i\sim_{\xx} j$.  Let $\nc_{n}$ denote the set of noncrossing set partitions of $n$.

Set partitions can be partially ordered as follows.  Let $\xx,\xx'\in\Pi_{n}$, and say that $\xx=\{B_{1},B_{2},\ldots,B_{s}\}$ and $\xx'=\{B'_{1},B'_{2},\ldots,B'_{s'}\}$.  We have $\xx\dref\xx'$ if and only if for each $i\in[s]$ there exists $i'\in[s']$ such that $B_{i}\subseteq B'_{i'}$.  We call $\dref$ the \defn{dual refinement order}.  Figure~\ref{fig:partitions_4} shows for the poset $(\Pi_{4},\dref)$, in which the subposet $(\nc_{4},\dref)$ is highlighted.  We have omitted braces in the labeling of the vertices, and have separated blocks by vertical lines instead.

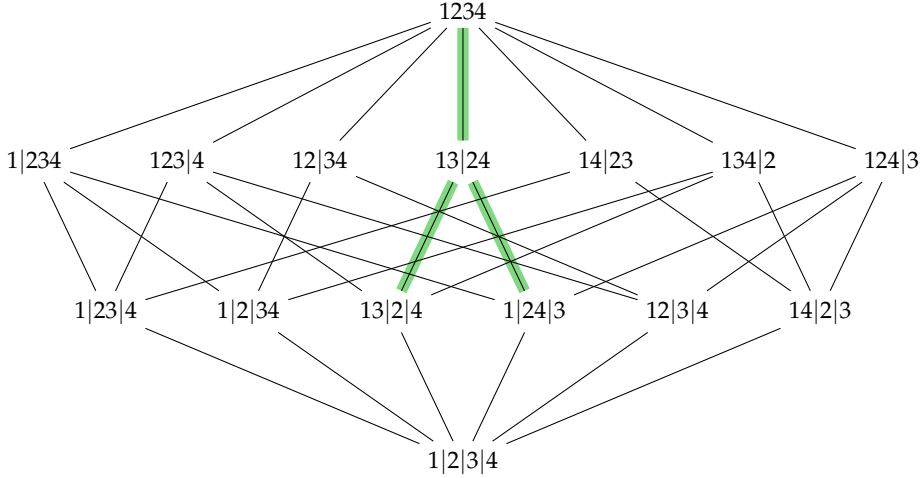
\begin{figure}
	\centering
	\begin{tikzpicture}\small
		\def\x{1.9};
		\def\y{2};
		\draw(3.5*\x,1*\y) node(n1){$1|2|3|4$};
		\draw(1*\x,2*\y) node(n2){$1|23|4$};
		\draw(2*\x,2*\y) node(n3){$1|2|34$};
		\draw(3*\x,2*\y) node(n4){$13|2|4$};
		\draw(4*\x,2*\y) node(n5){$1|24|3$};
		\draw(5*\x,2*\y) node(n6){$12|3|4$};
		\draw(6*\x,2*\y) node(n7){$14|2|3$};
		\draw(.5*\x,3*\y) node(n8){$1|234$};
		\draw(1.5*\x,3*\y) node(n9){$123|4$};
		\draw(2.5*\x,3*\y) node(n10){$12|34$};
		\draw(3.5*\x,3*\y) node(n11){$13|24$};
		\draw(4.5*\x,3*\y) node(n12){$14|23$};
		\draw(5.5*\x,3*\y) node(n13){$134|2$};
		\draw(6.5*\x,3*\y) node(n14){$124|3$};
		\draw(3.5*\x,4*\y) node(n15){$1234$};
		\draw(n1) -- (n2);
		\draw(n1) -- (n3);
		\draw(n1) -- (n4);
		\draw(n1) -- (n5);
		\draw(n1) -- (n6);
		\draw(n1) -- (n7);
		\draw(n2) -- (n8);
		\draw(n2) -- (n9);
		\draw(n2) -- (n12);
		\draw(n3) -- (n8);
		\draw(n3) -- (n10);
		\draw(n3) -- (n13);
		\draw(n4) -- (n9);
		\draw(n4) -- (n11);
		\draw(n4) -- (n13);
		\draw(n5) -- (n8);
		\draw(n5) -- (n11);
		\draw(n5) -- (n14);
		\draw(n6) -- (n9);
		\draw(n6) -- (n10);
		\draw(n6) -- (n14);
		\draw(n7) -- (n12);
		\draw(n7) -- (n13);
		\draw(n7) -- (n14);
		\draw(n8) -- (n15);
		\draw(n9) -- (n15);
		\draw(n10) -- (n15);
		\draw(n11) -- (n15);
		\draw(n12) -- (n15);
		\draw(n13) -- (n15);
		\draw(n14) -- (n15);
		\begin{pgfonlayer}{background}
			\draw[opacity=.5,line width=.15cm,green!70!black](n4) -- (n11);
			\draw[opacity=.5,line width=.15cm,green!70!black](n5) -- (n11);
			\draw[opacity=.5,line width=.15cm,green!70!black](n11) -- (n15);
		\end{pgfonlayer}
	\end{tikzpicture}
	\caption{The poset $(\Pi_{4},\dref)$.  The non-highlighted edges induce the subposet $(\nc_{4},\dref)$.}
	\label{fig:partitions_4}
\end{figure}

The posets $(\Pi_{n},\dref)$ and $(\nc_{n},\dref)$ are in fact lattices, and we can explicitly describe the meet and join operations.  The meet of two set partitions $\xx,\xx'\in\Pi_{n}$ is 
\begin{equation}\label{eq:partition_meet}
	\xx\wedge_{\Pi}\xx' = \{B\cap B'\mid B\in\xx, B'\in\xx',\;\text{and}\;B\cap B'\neq\emptyset\}.
\end{equation}
In order to describe the join of $\xx$ and $\xx'$, consider the bipartite graph 
\begin{displaymath}
	\mathbf{P}_{\xx,\xx'}=\bigl([n]\uplus(\xx\cup\xx'),E\bigr), 
\end{displaymath}
where $(v_{1},v_{2})\in E$ if and only if $v_{1}\in[n],v_{2}\in(\xx\cup\xx')$, and $v_{1}\in v_{2}$.  We have
\begin{equation}\label{eq:partition_join}
	\xx\vee_{\Pi}\xx' = \bigl\{C\cap[n]\mid C\;\text{is a connected component of}\;\mathbf{P}_{\xx,\xx'}\bigr\}.
\end{equation}

\begin{example}\label{ex:partitions_1}
	Let 
	\begin{displaymath}
		\xx=\bigl\{\{1\},\{2\},\{4\},\{3,5,7,8\},\{6\}\bigr\}\quad\text{and}\quad\xx'=\bigl\{\{1,3\},\{2,4\},\{5,6,8\},\{7\}\bigr\}.
	\end{displaymath}
	We observe that $\xx$ is non-crossing, while $\xx'$ is not, since $1\sim_{\xx'}3$ and $2\sim_{\xx'}4$, but $1\not\sim_{\xx'}2$.  Their meet is 
	\begin{displaymath}
		\xx\wedge_{\Pi}\xx' = \bigl\{\{1\},\{2\},\{3\},\{4\},\{5,8\},\{6\},\{7\}\bigr\}.
	\end{displaymath}
	The graph $\mathbf{P}_{\xx,\xx'}$ is\\
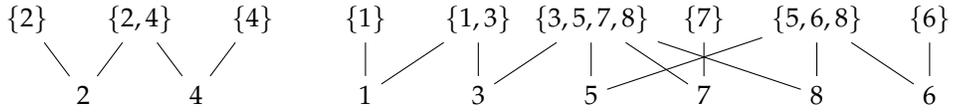
	\begin{center}\begin{tikzpicture}
		\def\x{1.5};
		\def\y{1};
		\draw(1.5*\x,1*\y) node(n2){$2$};
		\draw(2.5*\x,1*\y) node(n4){$4$};
		\draw(4*\x,1*\y) node(n1){$1$};
		\draw(5*\x,1*\y) node(n3){$3$};
		\draw(6*\x,1*\y) node(n5){$5$};
		\draw(8*\x,1*\y) node(n8){$8$};
		\draw(7*\x,1*\y) node(n7){$7$};
		\draw(9*\x,1*\y) node(n6){$6$};
		\draw(1*\x,2*\y) node(m2){$\{2\}$};
		\draw(2*\x,2*\y) node(m24){$\{2,4\}$};
		\draw(3*\x,2*\y) node(m4){$\{4\}$};
		\draw(4*\x,2*\y) node(m1){$\{1\}$};
		\draw(5*\x,2*\y) node(m13){$\{1,3\}$};
		\draw(6*\x,2*\y) node(m3578){$\{3,5,7,8\}$};
		\draw(8*\x,2*\y) node(m568){$\{5,6,8\}$};
		\draw(7*\x,2*\y) node(m7){$\{7\}$};
		\draw(9*\x,2*\y) node(m6){$\{6\}$};
		\draw(n1) -- (m1);
		\draw(n1) -- (m13);
		\draw(n2) -- (m2);
		\draw(n2) -- (m24);
		\draw(n3) -- (m13);
		\draw(n3) -- (m3578);
		\draw(n4) -- (m4);
		\draw(n4) -- (m24);
		\draw(n5) -- (m568);
		\draw(n5) -- (m3578);
		\draw(n6) -- (m6);
		\draw(n6) -- (m568);
		\draw(n7) -- (m7);
		\draw(n7) -- (m3578);
		\draw(n8) -- (m568);
		\draw(n8) -- (m3578);
	\end{tikzpicture}\end{center}
	which implies $\xx\vee_{\Pi}\xx'=\bigl\{\{1,3,5,6,7,8\},\{2,4\}\bigr\}$.
\end{example}

For $\xx\in\Pi_n$ denote by $\overline{\xx}$ the \defn{noncrossing closure} of $\xx$, which is defined by successively joining crossing blocks.  It is immediate that $\xx\dref\overline{\xx}$, and \cite{kreweras72sur}*{Th{\'e}or{\`e}me~1} states that $\overline{\xx}$ is the smallest noncrossing partition (weakly) above $\xx$.  The meet of two noncrossing set partitions $\xx,\xx'\in\nc_{n}$ is then
\begin{equation}\label{eq:nc_meet}
	\xx\wedge_{\nc}\xx' = \xx\wedge_{\Pi}\xx',
\end{equation}
while their join is
\begin{equation}\label{eq:nc_join}
	\xx\vee_{\nc}\xx' = \overline{\xx\vee_{\Pi}\xx'}.
\end{equation}

\begin{example}\label{ex:partitions_2}
	Let $\xx'$ be the crossing set partition from Example~\ref{ex:partitions_1}.  We obtain 
	\begin{displaymath}
		\overline{\xx'}=\bigl\{\{1,2,3,4\},\{5,6,8\},\{7\}\bigr\},
	\end{displaymath}
	and $\xx\wedge_{\nc}\overline{\xx'}=\xx\wedge_{\Pi}\xx'$ and $\xx\vee_{\nc}\overline{\xx'}=\one$.
\end{example}

Let us summarize this in a theorem.

\begin{theorem}[Folklore, \cite{kreweras72sur}*{Th{\'e}or{\`e}mes~2~and~3}]\label{thm:partitions}
	For $n\geq 1$, the posets $(\Pi_{n},\dref)$ and $(\nc_{n},\dref)$ are graded lattices.  The rank of a (noncrossing) set partition is given by $n$ minus the number of its blocks.
\end{theorem}

For $i\in[n]$ define $\xx_{i}$ to be the noncrossing partition with the unique non-singleton block $[i-1]\cup\{n\}$.  We thereby understand $\xx_{1}=\zero$ and $\xx_{n}=\one$.  It follows that 
\begin{equation}\label{eq:good_chain}
	C=\{\xx_{1},\xx_{2},\ldots,\xx_{n}\}
\end{equation}
is a maximal chain in $(\nc_{n},\dref)$

\begin{proposition}\label{prop:xi_modular}
	For $i\in[n]$ the element $\xx_{i}$ is left-modular in $(\nc_{n},\dref)$.  
\end{proposition}
\begin{proof}
	Let $X=[i-1]\cup\{n\}$ be the unique non-singleton block of $\xx_{i}$, and let $\zz\in\nc_{n}$.  
	
	\medskip
	
	We show that $\xx_{i}M\zz$.  Pick $\yy\dref\zz$, and let $B$ be a block of $\yy$.  There exists a unique block $B'$ of $\zz$ with $B\subseteq B'$.  Let $A=B'\cap X$.  We distinguish two cases.
	
	(i) $B\cap X=\emptyset$.  It follows that $B$ is a block of $\yy\vee_{\nc}\xx_{i}$, and it is thus a block of $(\yy\vee_{\nc}\xx_{i})\wedge_{\nc}\zz$, too.  In $\xx_{i}\wedge_{\nc}\zz$ we see that $A$ is a block, while $B'\setminus A$ is split into singleton blocks.  By assumption $B\subseteq (B'\setminus A)$, and we conclude that $B$ is a block of $\yy\vee_{\nc}(\xx_{i}\wedge_{\nc}\zz)$.
	
	(ii) $B\cap X\neq\emptyset$.  It follows that $B\cup X$ is a block of $\yy\vee_{\nc}\xx_{i}$, and that therefore $A\cup B$ is a block of $(\yy\vee_{\nc}\xx_{i})\wedge_{\nc}\zz$.  In $\xx_{i}\wedge_{\nc}\zz$ we see that $A$ is a block, while $B'\setminus A$ is split into singleton blocks.  By assumption $B\cap A\neq\emptyset$, and we thus obtain that $A\cup B$ is a block of $\yy\vee_{\nc}(\xx_{i}\wedge_{\nc}\zz)$.  
	
%
%
%
%
%
\end{proof}

\begin{corollary}\label{cor:nc_supersolvable}
	The chain in \eqref{eq:good_chain} is a left-modular chain in $(\nc_{n},\dref)$, which is thus a supersolvable lattice. 
\end{corollary}
\begin{proof}
	Proposition~\ref{prop:xi_modular} implies that every element in \eqref{eq:good_chain} is left-modular, and Theorem~\ref{thm:partitions} implies that $(\nc_{n},\dref)$ is graded.  In view of Theorem~\ref{thm:graded_left_modular} we conclude that $(\nc_{n},\dref)$ is supersolvable.
\end{proof}

The fact that $(\nc_{n},\dref)$ is supersolvable was established before in \cite{hersh99decomposition}*{Theorem~4.3.2}.  

\begin{corollary}\label{cor:nc_shellable}
	For $n\geq 1$, the lattice $(\nc_{n},\dref)$ is EL-shellable.
\end{corollary}
\begin{proof}
	This follows from Theorem~\ref{thm:left_modular_snelling} and Corollary~\ref{cor:nc_supersolvable}.
\end{proof}

The fact that $(\nc_{n},\dref)$ is EL-shellable was established before in \cite{bjorner80shellable}*{Example~2.9}.

\section{A Subposet of $(\nc_{n},\dref)$}
	\label{sec:pf_element_poset}
Let us define two subsets $L_{1},L_{2}\subseteq\nc_{n}$ by
\begin{align*}
	L_{1} & = \bigl\{\xx\in\nc_{n}\mid \{n-1,n\}\in\xx\bigr\},\\
	L_{2} & = \bigl\{\xx\in\nc_{n}\mid 1\sim_{\xx}n-1\;\text{and}\;\{n\}\in\xx\bigr\}.
\end{align*}
Finally, for $n\geq 3$ define
\begin{equation}\label{eq:nc_desired}
	\pe_{n} = \nc_{n}\setminus\bigl(L_{1}\cup L_{2}\bigr).
\end{equation}

\begin{lemma}[\cite{bruce16decomposition}]
	We have $\bigl\lvert\pe_{3}\bigr\rvert=3$, and for $n\geq 4$ we have 
	\begin{displaymath}
		\Bigl\lvert\pe_{n}\Bigr\rvert = \left(\frac{5}{n+1}+\frac{9}{n-3}\right)\binom{2n-4}{n-4},
	\end{displaymath}
	which is \cite{sloane}*{A071718} with offset $2$.
\end{lemma}
\begin{proof}
	Define the $n^{\text{th}}$ Catalan number to be $\cat(n)=\tfrac{1}{n+1}\tbinom{2n}{n}$.  It was observed in \cite{bruce16decomposition} that
	\begin{displaymath}
		\Bigl\lvert\pe_{n}\Bigr\rvert = \cat(n)-2\cat(n-2).
	\end{displaymath}
	We can therefore immediately verify the claim for $n=3$.  For $n\geq 4$, we obtain
	\begin{align*}
		\Bigl\lvert\pe_{n}\Bigr\rvert & = \cat(n)-2\cat(n-2)\\
		& = \frac{1}{n+1}\binom{2n}{n}-\frac{2}{n-1}\binom{2n-4}{n-2}\\
		& = \left(\frac{4(2n-1)(2n-3)}{(n+1)(n-2)(n-3)}-\frac{2n}{(n-2)(n-3)}\right)\binom{2n-4}{n-4}\\
		& = \left(\frac{14n^{2}-34n+12}{(n+1)(n-2)(n-3)}\right)\binom{2n-4}{n-4}\\
		& = \left(\frac{14n-6}{(n+1)(n-3)}\right)\binom{2n-4}{n-4}\\
		& = \left(\frac{5}{n+1}+\frac{9}{n-3}\right)\binom{2n-4}{n-4}.
	\end{align*}
\end{proof}

\subsection{$(\pe_{n},\dref)$ is a Supersolvable Lattice}
	\label{sec:pf_element_structure}
Let us now investigate a few properties of the poset $(\pe_{n},\dref)$.  Our first main result establishes that this poset is in fact a lattice. 

\begin{theorem}\label{thm:pf_element_lattice}
	For $n\geq 3$, the poset $(\pe_{n},\dref)$ is a lattice.
\end{theorem}
\begin{proof}
	Let $\xx,\xx'\in\pe_{n}$.  Let $\ww=\xx\wedge_{\nc}\xx'$, and write $\ww=\{B_{1},B_{2},\ldots,B_{s}\}$.  If $\ww\in\pe_{n}$, define $\xx\wedge_{\pe}\xx'=\ww$.  If $\ww\notin\pe_{n}$, then there are two options.   
	
	(i) $\{n-1,n\}\in\ww$.  Without loss of generality say that $B_{s}=\{n-1,n\}$.  Define $\ww'=\bigl\{B_{1},B_{2},\ldots,B_{s-1},\{n-1\},\{n\}\bigr\}$.  Then, $\ww'\in\pe_{n}$, and $\ww'\dref\ww$, which in particular implies that $\ww'\dref\xx$ and $\ww'\dref\xx'$.  Let $\zz\in\pe_{n}$ with $\zz\dref\xx$ and $\zz\dref\xx'$.  We must thus have $\zz\dref\ww$, and $\{n-1,n\}\notin\zz$, which implies $\{n-1\},\{n\}\in\zz$ and every block of $\zz$ is contained in some $B_{i}$ for $i\in[s]$.  It follows that $\zz\dref\ww'$.  We thus put $\xx\wedge_{\pe}\xx'=\ww'$ for this case.
	
	(ii) $\{n\}\in\ww$ and $1\sim_{\ww}n-1$.  Without loss of generality we can assume that $B_{s}=\{n\}$.  By definition we must have $1\sim_{\xx}n-1$ and $1\sim_{\xx'}n-1$.  Since $\xx,\xx'\in\pe_{n}$ we conclude that there are indices $i\neq j$ with $i\sim_{\xx}n$ and $j\sim_{\xx'}n$.  Since $\{n\}\in\ww$ we conclude $1<i,j<n-1$, which contradicts $\xx,\xx'\in\nc_{n}$.  It follows that this case cannot occur.
	
	\medskip
	
	Now let $\ww=\xx\vee_{\nc}\xx'$, and write $\ww=\{B_{1},B_{2},\ldots,B_{s}\}$.  If $\ww\in\pe_{n}$, define $\xx\vee_{\pe}\xx'=\ww$.  If $\ww\notin\pe_{n}$, then there are two options again.  
	
	(i) $\{n-1,n\}\in\ww$.  In view of \eqref{eq:partition_join} we conclude $\{n-1,n\}\in\xx,\xx'$, which contradicts $\xx,\xx'\in\pe_{n}$.  It follows that this case cannot occur.
	
	(ii) $\{n\}\in\ww$ and $1\sim_{\ww}n-1$.  Without loss of generality let $1,n-1\in B_{1}$, and let $B_{s}=\{n\}$.  Define $\ww'=\{B_{1}\cup B_{s},B_{2},\ldots,B_{s-1}\}$.  We then have $\ww\dref\ww'$, and consequently $\xx\dref\ww'$ and $\xx'\dref\ww'$.  Let $\zz\in\pe_{n}$ with $\xx\dref\zz$ and $\xx'\dref\zz$.  Again by \eqref{eq:partition_join} we conclude $\{n\}\in\xx,\xx'$, and since $\xx,\xx'\in\pe_{n}$ we see that $1\not\sim_{\xx}n-1$ and $1\not\sim_{\xx'}n-1$.  Since $1\sim_{\ww}n-1$ there must be $i\in[n]$ with $1\sim_{\xx}i$ and $i\sim_{\xx'}n-1$.  We thus conclude $1\sim_{\zz}n-1$, and since $\zz\in\pe_{n}$ we further conclude $n-1\sim_{\zz}n$.  This implies $\ww'\dref\zz$.  We thus put $\xx\vee_{\pe}\xx'=\ww'$ for this case.
\end{proof}

\begin{lemma}\label{lem:pf_element_graded}
	For $n\geq 3$, the lattice $(\pe_{n},\dref)$ is graded.
\end{lemma}
\begin{proof}
	Let $\xx,\yy\in\pe_{n}$ with $\xx\lessdot_{\text{\normalfont dref}}\yy$ in $(\pe_{n},\dref)$.  Assume that there is $\zz\in\nc_{n}$ with $\xx<_{\text{\normalfont dref}}\zz<_{\text{\normalfont dref}}\yy$.  It follows that $\zz\in\nc_{n}\setminus\pe_{n}$.  There are two cases.  
	
	(i) $\{n-1,n\}$ is a block of $\zz$.  Since $\{n-1,n\}$ is neither a block of $\xx$, nor of $\yy$, it must be that $n-1$ and $n$ constitute singleton blocks in $\xx$ and there is some $j\in[n-2]$ and some block $B$ of $\yy$ containing $\{j,n-1,n\}$.  Consider the partition $\ww$ that has all blocks of $\yy$ except that $B$ is replaced by the two blocks $B\setminus\{n-1\}$ and $\{n-1\}$.  Since $\yy\in\pe_{n}\subseteq\nc_{n}$ we conclude that $\ww\in\nc_{n}$, and we have $\ww\lessdot_{\text{\normalfont dref}}\yy$.  By construction, $\ww\in\pe_{n}$.  It follows further from $\xx\dref\yy$ that $\xx<_{\text{\normalfont dref}}\ww$ (since $n-1$ and $n$ constitute singleton blocks of $\xx$).  This is a contradiction to $\xx\lessdot_{\text{\normalfont dref}}\yy$ in $(\pe_{n},\dref)$. 
	
	(ii) $\{n\}$ is a block of $\zz$ and $1\sim_{\zz}n-1$.  It follows that $1\sim_{\yy}n-1$, which forces $n-1\sim_{\yy}n$.  Moreover, it follows that $\{n\}$ must be a block of $\xx$, which implies that $1\not\sim_{\xx}n-1$.  Let $B$ be the block of $\xx$ containing $1$.  Consider the partition $\ww$ that consists of all the blocks of $\xx$ except that $B$ is replaced by $B\cup\{n\}$.  Then, $\xx\in\nc_{n}$ implies $\ww\in\pe_{n}$.  Moreover, $\xx\lessdot_{\text{\normalfont dref}}\ww<_{\text{\normalfont dref}}y$, which is a contradiction to $\xx\lessdot_{\text{\normalfont dref}}\yy$ in $(\pe_{n},\dref)$. 
\end{proof}

It follows by definition that the chain \eqref{eq:good_chain} belongs to $(\pe_{n},\dref)$.  It is our next goal to show that this chain is also left-modular in $(\pe_{n},\dref)$.  We first prove an auxiliary result.

\begin{proposition}\label{prop:equal_join_meet}
	For $i\in[n]$ and $\yy\in\pe_{n}$ we have $\xx_{i}\wedge_{\pe}\yy=\xx_{i}\wedge_{\nc}\yy$ and $\xx_{i}\vee_{\pe}\yy=\xx_{i}\vee_{\nc}\yy$.
\end{proposition}
\begin{proof}
	Let $\yy\in\pe_{n}$.  If $\xx_{i}\wedge_{\pe}\yy<_{\text{\normalfont dref}}\xx_{i}\wedge_{\nc}\yy$, then it follows from the proof of Theorem~\ref{thm:pf_element_lattice} that there exists a block $B$ of $\xx_{i}$ with $\{n-1,n\}\subseteq B$.  By definition this forces $i=n$, so that $\xx_{i}$ is the full partition.  In particular $\yy\dref\xx_{i}$, which yields the contradiction $\yy=\xx_{i}\wedge_{\pe}\yy<_{\text{\normalfont dref}}\xx_{i}\wedge_{\nc}\yy=\yy$.
	
	\medskip
	
	If $\xx_{i}\vee_{\nc}\yy<_{\text{\normalfont dref}}\xx_{i}\vee_{\pe}\yy$, then it follows from the proof of Theorem~\ref{thm:pf_element_lattice} that $\{n\}$ is a block of $\xx_{i}$.  By definition, this forces $i=1$, so that $\xx_{i}$ is the discrete partition.  In particular $\xx_{i}\dref\yy$, which yields the contradiction $\yy=\xx_{i}\vee_{\nc}\yy<_{\text{\normalfont dref}}\xx_{i}\vee_{\pe}\yy=\yy$.
\end{proof}

\begin{proposition}\label{prop:pf_element_left_modular}
	For $n\geq 3$, the chain in \eqref{eq:good_chain} is left-modular in $(\pe_{n},\dref)$.
\end{proposition}
\begin{proof}
	The elements $\xx_{1}$ and $\xx_{n}$ are the least and the greatest element of $(\pe_{n},\dref)$, so they are trivially left-modular.  Let us therefore assume that $i\in\{2,3,\ldots,n-1\}$.  In particular, $n-1\not\sim_{\xx_{i}}n$ and $\{n\}$ is not a block of $\xx_{i}$.  Let $\zz\in\pe_{n}$.

	\medskip
	
	We show that $\xx_{i}M\zz$ holds in $(\pe_{n},\dref)$.  Let $\yy\in\pe_{n}$ with $\yy\dref\zz$.  Proposition~\ref{prop:equal_join_meet} implies that $\qq=\yy\vee_{\pe}\xx_{i}=\yy\vee_{\nc}\xx_{i}$.  Assume that $\qq\wedge_{\pe}\zz\neq\qq\wedge_{\nc}\zz$.  The proof of Theorem~\ref{thm:pf_element_lattice} implies that this can only happen if $\{n-1,n\}$ is a block of $\qq\wedge_{\nc}\zz$.  For this to happen, we need $n-1\sim_{\qq}n$, which forces the existence of some $j\in[i-1]\cup\{n\}$ with $j\sim_{\yy}n-1$.  If $j<i$, then we obtain the contradiction that $\qq\wedge_{\nc}\zz$ has a block containing $\{j,n-1,n\}$ since $i\leq n-1$.  We thus have $j=n$.  Since $i>1$ we see that $\qq$ has a block containing $\{i-1,n-1,n\}$, which forces $\zz$ to contain the block $\{n-1,n\}$; a contradiction to $\zz\in\pe_{n}$.  We therefore have
	\begin{equation}\label{eq:lmod_pf_1}
		(\yy\vee_{\pe}\xx_{i})\wedge_{\pe}\zz = (\yy\vee_{\nc}\xx_{i})\wedge_{\nc}\zz.
	\end{equation}
	
	On the other hand, Proposition~\ref{prop:equal_join_meet} also implies that $\qq'=\xx_{i}\wedge_{\pe}\zz=\xx_{i}\wedge_{\nc}\zz$.  Assume that $\yy\vee_{\pe}\qq'\neq\yy\vee_{\nc}\qq'$.  The proof of Theorem~\ref{thm:pf_element_lattice} implies that this can only happen if $\{n\}$ is a block of $\yy\vee_{\nc}\qq'$ and $1\sim_{\yy\vee_{\nc}\qq'}n-1$.  By definition of the join, $\{n\}$ must be a block of both $\yy$ and $\qq'$.  Since $i<n$ we see that $\{n-1\}$ is a singleton block in $\qq'$, which forces $1\sim_{\yy}n-1$; a contradiction to $\yy\in\pe_{n}$.  We therefore have
	\begin{equation}\label{eq:lmod_pf_2}
		\yy\vee_{\pe}(\xx_{i}\wedge_{\pe}\zz) = \yy\vee_{\nc}(\xx_{i}\wedge_{\nc}\zz).
	\end{equation}
	Proposition~\ref{prop:xi_modular} implies the equality of the right-hand sides of \eqref{eq:lmod_pf_1} and \eqref{eq:lmod_pf_2}, which implies $\xx_{i}M\zz$ in $(\pe_{n},\dref)$.
\end{proof}

We now conclude the proof of Theorem~\ref{thm:pf_element_supersolvable}.

\begin{proof}[Proof of Theorem~\ref{thm:pf_element_supersolvable}]
	It follows from Theorem~\ref{thm:pf_element_lattice}, Lemma~\ref{lem:pf_element_graded}, and Proposition~\ref{prop:pf_element_left_modular} that $(\pe_{n},\dref)$ is a graded left-modular lattice.  Theorem~\ref{thm:graded_left_modular} implies then that it is supersolvable.
\end{proof}

\begin{corollary}\label{cor:pf_element_shellable}
	For $n\geq 3$, the lattice $(\pe_{n},\dref)$ is EL-shellable.
\end{corollary}
\begin{proof}
	This follows from Theorems~\ref{thm:pf_element_supersolvable} and \ref{thm:left_modular_snelling}.
\end{proof}

Figure~\ref{fig:pf_element_4} shows $(\pe_{4},\dref)$ together with the EL-labeling coming from the left-modular chain in \eqref{eq:good_chain}.  The unique rising maximal chain from $\zero$ to $\one$ is highlighted.

\begin{figure}
	\centering
	\begin{tikzpicture}\small
		\def\x{2};
		\def\y{2};
		\draw(2.5*\x,1*\y) node(n1){$1|2|3|4$};
		\draw(1*\x,2*\y) node(n2){$1|24|3$};
		\draw(2*\x,2*\y) node(n3){$1|23|4$};
		\draw(3*\x,2*\y) node(n4){$12|3|4$};
		\draw(4*\x,2*\y) node(n5){$14|2|3$};
		\draw(1*\x,3*\y) node(n6){$1|234$};
		\draw(2*\x,3*\y) node(n7){$124|3$};
		\draw(3*\x,3*\y) node(n8){$14|23$};
		\draw(4*\x,3*\y) node(n9){$134|2$};
		\draw(2.5*\x,4*\y) node(n10){$1234$};
		\draw(n1) -- (n2) node[fill=white] at (1.75*\x,1.5*\y){\tiny $2$};
		\draw(n1) -- (n3) node[fill=white] at (2.25*\x,1.5*\y){\tiny $3$};
		\draw(n1) -- (n4) node[fill=white] at (2.75*\x,1.5*\y){\tiny $2$};
		\draw(n1) -- (n5) node[fill=white] at (3.25*\x,1.5*\y){\tiny $1$};
		\draw(n2) -- (n6) node[fill=white] at (1*\x,2.5*\y){\tiny $3$};
		\draw(n2) -- (n7) node[fill=white] at (1.33*\x,2.33*\y){\tiny $1$};
		\draw(n3) -- (n6) node[fill=white] at (1.67*\x,2.33*\y){\tiny $2$};
		\draw(n3) -- (n8) node[fill=white] at (2.33*\x,2.33*\y){\tiny $1$};
		\draw(n4) -- (n7) node[fill=white] at (2.67*\x,2.33*\y){\tiny $1$};
		\draw(n5) -- (n7) node[fill=white] at (3*\x,2.5*\y){\tiny $2$};
		\draw(n5) -- (n8) node[fill=white] at (3.5*\x,2.5*\y){\tiny $3$};
		\draw(n5) -- (n9) node[fill=white] at (4*\x,2.5*\y){\tiny $3$};
		\draw(n6) -- (n10) node[fill=white] at (1.75*\x,3.5*\y){\tiny $1$};
		\draw(n7) -- (n10) node[fill=white] at (2.25*\x,3.5*\y){\tiny $3$};
		\draw(n8) -- (n10) node[fill=white] at (2.75*\x,3.5*\y){\tiny $2$};
		\draw(n9) -- (n10) node[fill=white] at (3.25*\x,3.5*\y){\tiny $2$};
		\begin{pgfonlayer}{background}
			\draw[opacity=.5,line width=.15cm,green!70!black](n1) -- (n5) -- (n7) -- (n10);
		\end{pgfonlayer}
	\end{tikzpicture}
	\caption{The lattice $(\pe_{4},\dref)$.  The highlighted chain is \eqref{eq:good_chain}, and the labeling is the one defined in \eqref{eq:left_modular_labeling}.}
	\label{fig:pf_element_4}
\end{figure}
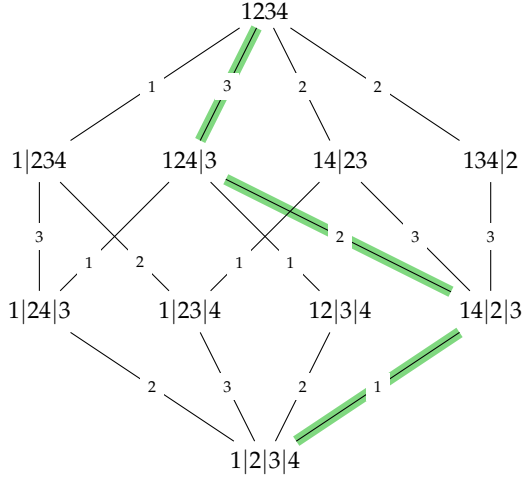

%
%

\subsection{The M{\"o}bius Function of $(\pe_{n},\dref)$}
	\label{sec:pf_element_mobius}
In this section we determine the value of the M{\"o}bius function of $(\pe_{n},\dref)$ between $\zero$ and $\one$.  Recall that the \defn{M{\"o}bius function} of a poset $\LL=(L,\leq)$ is defined recursively by
\begin{equation}\label{eq:mobius_function}
	\mu_{\LL}(x,y) = \begin{cases}1, & \text{if}\;x=y,\\-\sum_{x<z\leq y}{\mu(z,y)}, & \text{if}\;x<y,\\0, & \text{otherwise}\end{cases}
\end{equation}
for all $x,y\in L$.  It was shown in \cite{blass97mobius} that in a lattice $\LL$, we can compute the value $\mu_{\LL}(\hat{0},x)$ for any $x\in L$ by summing over the NBB bases for $x$.  Let us recall the necessary concepts.  Let $\AA$ denote the set of atoms of $\LL$, and let $\trianglelefteq$ be an arbitrary partial order on $\AA$.  A set $X\subseteq\AA$ is \defn{bounded below} (or \defn{BB} for short) if for every $d\in X$ there exists some $a\in\AA$ such that $a\triangleleft d$ and $a<\bigvee X$.  A set $X\subseteq\AA$ is \defn{NBB} if none of its nonempty subsets is BB.  If $X$ is NBB and $\bigvee X=x$, then $X$ is a \defn{NBB base} for x.  We have the following result.

\begin{theorem}[\cite{blass97mobius}*{Theorem~1.1}]\label{thm:nbb_base_mobius}
	Let $\LL=(L,\leq)$ be a finite lattice, and let $\trianglelefteq$ be any partial order on the atoms of $\LL$.  For $x\in L$ we have 
	\begin{displaymath}
		\mu_{\LL}(\hat{0},x) = \sum_{X}{(-1)^{\lvert X\rvert}},
	\end{displaymath}
	where the sum is over all NBB bases for $x$ with respect to $\trianglelefteq$. 
\end{theorem}

In the remainder of this section we give a combinatorial model for the NBB bases of $\one$ in $(\pe_{n},\dref)$ with respect to a suitable partial order on its atoms, and conclude Theorem~\ref{thm:pf_element_mobius}.

For $i,j\in[n]$ with $i<j$, define $\aa_{i,j}$ to be the set partition whose unique non-singleton block is $\{i,j\}$.  The set $\AA_{n}=\{\aa_{i,j}\mid 1\leq i<j\leq n\}$ is the set of all atoms of $\{\nc_{n},\dref)$.  The set $\bar{\AA}_{n}=\AA_{n}\setminus\{\aa_{1,n-1},\aa_{n-1,n}\}$ is then the set of atoms of $(\pe_{n},\dref)$.  Consider the partition of $\AA_{n}$ given by 
\begin{displaymath}
	A_{i}=\{\aa\in\AA_{n}\mid\aa\dref\xx_{i}\;\text{and}\;\aa\not\dref\xx_{i-1}\}
\end{displaymath}
for $i\in[n-1]$.  Let $\bar{A}_{i}$ be the restriction of $A_{i}$ to $\bar{\AA}_{n}$.  Define a partial order on $\AA_{n}$ by setting $\aa\trianglelefteq\aa'$ if and only if $\aa\in A_{i}$ and $\aa'\in A_{j}$ for $i<j$.  The poset $(\AA_{5},\trianglelefteq)$ is depicted in Figure~\ref{fig:atom_order_5}.

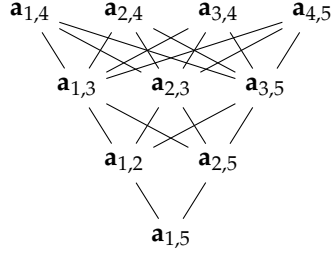
\begin{figure}
	\centering
	\begin{tikzpicture}
		\def\x{1.25};
		\def\y{1};
		\draw(2*\x,1*\y) node(n1){$\aa_{1,5}$};
		\draw(1.5*\x,2*\y) node(n2){$\aa_{1,2}$};
		\draw(2.5*\x,2*\y) node(n3){$\aa_{2,5}$};
		\draw(1*\x,3*\y) node(n4){$\aa_{1,3}$};
		\draw(2*\x,3*\y) node(n5){$\aa_{2,3}$};
		\draw(3*\x,3*\y) node(n6){$\aa_{3,5}$};
		\draw(.5*\x,4*\y) node(n7){$\aa_{1,4}$};
		\draw(1.5*\x,4*\y) node(n8){$\aa_{2,4}$};
		\draw(2.5*\x,4*\y) node(n9){$\aa_{3,4}$};
		\draw(3.5*\x,4*\y) node(n10){$\aa_{4,5}$};
		\draw(n1) -- (n2);
		\draw(n1) -- (n3);
		\draw(n2) -- (n4);
		\draw(n2) -- (n5);
		\draw(n2) -- (n6);
		\draw(n3) -- (n4);
		\draw(n3) -- (n5);
		\draw(n3) -- (n6);
		\draw(n4) -- (n7);
		\draw(n4) -- (n8);
		\draw(n4) -- (n9);
		\draw(n4) -- (n10);
		\draw(n5) -- (n7);
		\draw(n5) -- (n8);
		\draw(n5) -- (n9);
		\draw(n5) -- (n10);
		\draw(n6) -- (n7);
		\draw(n6) -- (n8);
		\draw(n6) -- (n9);
		\draw(n6) -- (n10);
	\end{tikzpicture}
	\caption{The poset $(\AA_{5},\trianglelefteq)$.}
	\label{fig:atom_order_5}
\end{figure}

\begin{lemma}\label{lem:atom_order_blocks}
	For $j\in[n-1]$ we have $A_{j}=\bigl\{\aa_{i,j}\mid 1\leq i<j\bigr\}\cup\bigl\{\aa_{j,n}\bigr\}$.  Moreover, we have $\bar{A}_{j}=A_{j}$ for $j\in[n-2]$, and $\bar{A}_{n-1}=A_{n-1}\setminus\{\aa_{1,n-1},\aa_{n-1,n}\}$.  
\end{lemma}
\begin{proof}
	Let $\aa_{i,j}\in\AA_{n}$ for $1\leq i<j\leq n$.  If $j<n$, then $\aa_{i,j}\dref\xx_{j}$, but $\aa_{i,j}\not\dref\xx_{j-1}$.  If $j=n$, then $\aa_{i,n}\dref\xx_{i}$, but $\aa_{i,n}\not\dref\xx_{i-1}$.
\end{proof}

Since we want to consider NBB bases in the two related posets $(\nc_{n},\dref)$ and $(\pe_{n},\dref)$, we use the prefixes ``$\nc$'' and ``$\pe$'' to indicate which lattice we consider.  Theorem~\ref{thm:pf_element_lattice} implies that for $\xx,\yy\in\pe_{n}$ we always have $\xx\vee_{\nc}\yy\dref\xx\vee_{\pe}\yy$.  Therefore, if $X\subseteq\bar{\AA}_{n}$ is $\nc$-BB, then it is automatically $\pe$-BB.

\begin{lemma}\label{lem:atoms_bb_1}
	If $\aa\vee_{\Pi}\aa'$ is crossing, then $\bigl\{\aa,\aa'\}$ is $\nc$-BB.
\end{lemma}
\begin{proof}
	Let $\aa_{i,j},\aa_{k,l}\in\AA_{n}$.  If $\aa_{i,j}\vee_{\Pi}\aa_{k,l}$ is crossing, then $i<k<j<l$, and the join $\aa_{i,j}\vee_{\nc}\aa_{k,l}$ has the unique non-singleton block $\{i,j,k,l\}$.  We distinguish two cases.  
	
	(i) If $l<n$, then Lemma~\ref{lem:atom_order_blocks} implies $\aa_{i,j}\in A_{j}$ and $\aa_{k,l}\in A_{l}$.  Since $j<l$ we obtain $\aa_{i,j}\triangleleft\aa_{k,l}$, and since $k<j$, Lemma~\ref{lem:atom_order_blocks} implies that $\aa_{i,k}\triangleleft\aa_{i,j}$.  We clearly have $\aa_{i,k}<_{\text{\normalfont dref}}\aa_{i,j}\vee_{\nc}\aa_{k,l}$, which implies that $\{\aa_{i,j},\aa_{k,l}\}$ is $\nc$-BB.  
	
	(ii) If $l=n$, then Lemma~\ref{lem:atom_order_blocks} implies $\aa_{i,j}\in A_{j}$ and $\aa_{k,n}\in A_{k}$.  Since $k<j$ we obtain $\aa_{k,n}\triangleleft\aa_{i,j}$, and since $i<k$, Lemma~\ref{lem:atom_order_blocks} implies that $\aa_{i,n}\triangleleft\aa_{k,n}$.  We clearly have $\aa_{i,n}<_{\text{\normalfont dref}}\aa_{i,j}\vee_{\pe}\aa_{k,n}$, which implies that $\{\aa_{i,j},\aa_{k,n}\}$ is $\nc$-BB.
\end{proof}

\begin{lemma}\label{lem:atoms_bb_2}
	If $\aa,\aa'\in A_{j}$ for $j\in[n-1]$, then $\{\aa,\aa'\}$ is $\nc$-BB.
\end{lemma}
\begin{proof}
	Let $\aa_{i,j},\aa_{k,j}\in A_{j}$.  Note that $\aa_{i,j}\vee_{\nc}\aa_{k,j}$ has the unique non-singleton block $\{i,k,j\}$.  There are again two cases.
	
	(i) If $i<j$ and $k<j$, then Lemma~\ref{lem:atom_order_blocks} implies $\aa_{i,k}\in A_{k}$, and thus $\aa_{i,k}\triangleleft\aa_{i,j}$ and $\aa_{i,k}\triangleleft\aa_{k,j}$.  We clearly have $\aa_{i,k}<_{\text{\normalfont dref}}\aa_{i,j}\vee_{\nc}\aa_{k,j}$, which implies that $\{\aa_{i,j},\aa_{k,j}\}$ is $\nc$-BB.
	
	(ii) If $i<j$ and $k>j$.  Lemma~\ref{lem:atom_order_blocks} implies that $k=n$, and that $\aa_{i,n}\in A_{i}$.  Therefore $\aa_{i,n}\triangleleft\aa_{i,j}$ and $\aa_{i,n}\triangleleft\aa_{j,n}$.  We clearly have $\aa_{i,n}<_{\text{\normalfont dref}}\aa_{i,j}\vee_{\nc}\aa_{j,n}$, which implies that $\{\aa_{i,j},\aa_{j,n}\}$ is $\nc$-BB.
\end{proof}

\begin{lemma}\label{lem:atoms_bb_3}
	Let $X\subseteq\bar{\AA}_{n}$ satisfy $\bigvee_{\pe}X=\one$.  If $\lvert X\rvert<n-1$, then $X$ is $\pe$-BB.
\end{lemma}
\begin{proof}
	Suppose that $\lvert X\rvert=k$.  Observe that if $X$ is a set of pairwise non-crossing atoms, then $\bigvee_{\nc}X=\bigvee_{\Pi}X$.  By \eqref{eq:partition_join} $\bigvee_{\Pi}X$ has exactly $n-k$ blocks.  Moreover, by Theorem~\ref{thm:pf_element_lattice} the number of blocks of $\bigvee_{\pe}X$ is either $n-k$ or $n-k-1$.  Since we assumed $\bigvee_{\pe}X=\one$, we conclude that $k\in\{n-2,n-1\}$.  Let $\zz=\bigvee_{\nc}X$.

	If $k=n-2$, then we conclude that $1\sim_{\zz}n-1$, and $\{n\}$ is a block of $\zz$.  It follows that $\aa_{1,n}\notin X$, which in view of Lemma~\ref{lem:atom_order_blocks} implies $\aa_{1,n}\triangleleft\aa$ for all $\aa\in X$.  Since $\aa_{1,n}<_{\text{\normalfont dref}}\one$, we conclude that $X$ is $\pe$-BB.
\end{proof}

Let us denote by $\BB_{n}$ the set of all $\nc$-NBB bases for $\one$, and let $\bar{\BB}_{n}$ denote the set of all $\pe$-NBB bases for $\one$.  By construction we have $\bar{\BB}_{n}\subseteq \BB_{n}$.  

\begin{corollary}\label{cor:nbb_bases_size}
	Every element of $\BB_{n}$ has cardinality $n-1$.  Consequently the same is true for the elements of $\bar{\BB}_{n}$.
\end{corollary}
\begin{proof}
	The claim for the cardinality of the elements in $\BB_{n}$ follows directly from \eqref{eq:nc_join} and Lemmas~\ref{lem:atoms_bb_1} and \ref{lem:atoms_bb_2}.  
	
	The claim for the cardinality of the elements in $\bar{\BB}_{n}$ can be verified directly using Lemmas~\ref{lem:atoms_bb_1}--\ref{lem:atoms_bb_3}.
\end{proof}

For the moment, let us focus on the elements of $\BB_{n}$.  In view of Corollary~\ref{cor:nbb_bases_size} these elements are certain maximal chains of $(\AA_{n},\trianglelefteq)$.  We can naturally associate a graph with $X\in\BB_{n}$ by connecting the vertices $i$ and $j$ if and only if $\aa_{i,j}\in X$.  Denote the resulting graph by $\tau(X)$.

\begin{lemma}\label{lem:nbb_base_tree}
	If $X\in\BB_{n}$, then $\tau(X)$ is a tree.
\end{lemma}
\begin{proof}
	Since ${\bigvee}_{\nc}X=\one$ it follows from \eqref{eq:partition_join} that $\tau(X)$ is connected.  Now suppose that $\tau(X)$ contains a cycle $C=(\aa_{i_{1},i_{2}},\aa_{i_{2},i_{3}},\ldots,\aa_{i_{s},i_{1}})$.  We then have $i_{1}<i_{2}<\cdots<i_{s}$, and $3\leq s<n$.  
	
	If $i_{s}<n$, then $\aa_{i_{s-1},i_{s}},\aa_{i_{1},i_{s}}\in A_{i_{s}}$, which contradicts Lemma~\ref{lem:atoms_bb_2}.  If $i_{s}=n$, then $\aa_{i_{s-2},i_{s-1}},\aa_{i_{s-1},i_{s}}\in A_{i_{s-1}}$, which contradicts Lemma~\ref{lem:atoms_bb_2}.
\end{proof}

Since $\aa_{1,n}$ is the least element in $(\AA_{n},\trianglelefteq)$ any of the trees in Lemma~\ref{lem:nbb_base_tree} contains an edge between $1$ and $n$.  

\begin{lemma}\label{lem:tree_separation}
	Let $X\in\BB_{n}$, and let $\tau(X)$ be the corresponding tree.  If we remove the edge between $1$ and $n$, we obtain two trees $\tau_{1}$ and $\tau_{2}$, where $\tau_{1}$ has vertex set $[k]$ and $\tau_{2}$ has vertex set $\{k+1,k+2,\ldots,n\}$ for some $k\in[n-1]$.  
\end{lemma}
\begin{proof}
	Suppose that $\tau_{1}$ and $\tau_{2}$ are the two trees obtained by removing the edge connecting $1$ and $n$ in $\tau(X)$.  The claim is certainly true for $n\leq 3$, so suppose that $n>3$.  Assume that there is a vertex $k$ in $\tau_{1}$ such that there exists $i\in[k-1]$ which is a vertex of $\tau_{2}$, and choose $k$ minimal with this property.  Since $\tau_{1}$ is a tree, there is a unique path from $1$ to $k$, and let $k'$ be the predecessor of $k$ along this path.  It follows that $\aa_{k',k}\in X$, and thus $k'<k$.  The minimality of $k$ implies that there is $l$ in $\{k'+1,k'+2,\ldots,k-1\}$ which is a vertex of $\tau_{2}$.  Let $l=l_{0}<l_{1}<\cdots<l_{s}=n$ denote the elements (in order) on the unique path from $l$ to $n$ in $\tau_{2}$.  Again by construction we have $\aa_{l_{i-1},l_{i}}\in X$ for $i\in[s]$.  Moreover, there exists a unique index $i\in[s]$ such that $l_{i-1}<k$ and $l_{i}>k$.  Then, however, Lemma~\ref{lem:atoms_bb_1} implies that $\{\aa_{k',k},\aa_{l_{i-1},l_{i}}\}$ is $\nc$-BB, which contradicts the fact that $X$ is an $\nc$-NBB base for $\one$.  This completes the proof.
\end{proof}

We say that the trees occurring as $\tau(X)$ for some $X\in\BB_{n}$ are \defn{noncrossing}.  Recall that the \defn{Catalan numbers} are defined by $\cat(n)=\tfrac{1}{n+1}\tbinom{2n}{n}$, and they satisfy the recurrence relation 
\begin{equation}\label{eq:catalan_recursion}
	\cat(n+1)=\sum_{k=0}^{n}{\cat(k)\cat(n-k)},
\end{equation}
with initial condition $\cat(0)=1$~\cite{segner61enumeratio}. 

\begin{corollary}\label{cor:tree_catalan}
	For $n\geq 1$ we have $\bigl\lvert\BB_{n}\bigr\rvert=\cat(n-1)$.
\end{corollary}
\begin{proof}
	Let $C_{n}=\bigl\lvert\BB_{n}\bigr\rvert$.  Lemma~\ref{lem:tree_separation} implies that $C_{n}=\sum_{k=1}^{n-1}{C_{k}C_{n-k}}$, and it is quickly verified that $C_{1}=1$.  Therefore the numbers $C_{n}$ and $\cat(n-1)$ satisfy the same recurrence relation and the same initial condition and must thus be equal.
\end{proof}

In view of Theorem~\ref{thm:nbb_base_mobius} we obtain the following well-known corollary.

\begin{corollary}[\cite{kreweras72sur}*{Th{\'e}or{\`e}me~6}]\label{cor:nc_mobius}
	For $n\geq 1$ we have 
	\begin{displaymath}	
		\mu_{(\nc_{n},\dref)}(\zero,\one)=(-1)^{n-1}\cat(n-1).
	\end{displaymath}
\end{corollary}

We are now ready to prove Theorem~\ref{thm:pf_element_mobius}.

\begin{proof}[Proof of Theorem~\ref{thm:pf_element_mobius}]
	In view of Corollary~\ref{cor:tree_catalan} it remains to determine the size of $\BB_{n}\setminus\bar{\BB}_{n}$.  Essentially this set consists of three types of elements; those that contain $\aa_{1,n-1}$, those that contain $\aa_{n-1,n}$, and those that (after removal of $\aa_{1,n}$) join to $\one$ in $(\pe_{n},\dref)$.  Since every element of $\BB_{n}$ contains $\aa_{1,n}$, Lemma~\ref{lem:nbb_base_tree} implies that $X\in\BB_{n}$ cannot contain both of $\aa_{1,n-1}$ and $\aa_{n-1,n}$.
	
	Let $\SM_{n}^{(1)}=\{X\in\BB_{n}\mid\aa_{1,n-1}\in X\}$ and $\SM_{n}^{(2)}=\{X\in\BB_{n}\mid\aa_{n-1,n}\in X\}$, and let 
	\begin{displaymath}
		\RR_{n} = \Bigl\{X\in\BB_{n}\mid{\bigvee}_{\pe}\bigl(X\setminus\{\aa_{1,n}\}\bigr)=\one\Bigr\}.
	\end{displaymath}
	By construction we have $\bar{\BB}_{n}=\BB_{n}\setminus\Bigl(\SM_{n}^{(1)}\cup\SM_{n}^{(2)}\cup\RR_{n}\Bigr)$. 
	
	The proof of Theorem~\ref{thm:pf_element_lattice} implies that for $X\in\RR_{n}$ the only vertex adjacent to $n$ in the corresponding tree $\tau(X)$ is $1$.  As a consequence $\SM_{n}^{(1)}\subseteq\RR_{n}$, and $\SM_{n}^{(2)}\cap\RR_{n}=\emptyset$.  It therefore suffices to determine the cardinalities of $\SM_{n}^{(2)}$ and $\RR_{n}$. 
	
	Let $X\in\SM_{n}^{(2)}$, and let $\tau(X)$ be the corresponding noncrossing tree.  Lemma~\ref{lem:tree_separation} implies that there is some $k\in[n-1]$ such that after removing the edge between $1$ and $n$ we are left with a noncrossing tree $\tau_{1}$ on vertex set $[k]$ and a noncrossing tree $\tau_{2}$ on vertex set $\{k+1,k+2,\ldots,n\}$ which has an edge between $n-1$ and $n$.   As a consequence, $k<n-1$ and we can view $\tau_{2}$ as a noncrossing tree on $n-k-1$ vertices.  We obtain $\Bigl\lvert\SM_{1}^{(2)}\Bigr\rvert=1$, and 
	\begin{displaymath}
		\Bigl\lvert\SM_{n}^{(2)}\Bigr\rvert = \sum_{k=1}^{n-2}{\Bigl\lvert\BB_{k}\Bigr\rvert\cdot\Bigl\lvert\BB_{n-k-1}\Bigr\rvert},
	\end{displaymath}
	which in view of \eqref{eq:catalan_recursion} implies $\Bigl\lvert\SM_{n}^{(2)}\Bigr\rvert=\cat(n-2)$. 
	
	Let $X\in\RR_{n}$.  We have seen already that in $\tau(X)$ the only edge adjacent to $n$ is $1$.  It follows that the elements of $\RR_{n}$ correspond bijectively to noncrossing trees on $n-1$ vertices.  Corollary~\ref{cor:tree_catalan} then implies that $\Bigl\lvert\RR_{n}\Bigr\rvert=\cat(n-2)$. 

	We thus obtain
	\begin{align*}
		\Bigl\lvert\bar{\BB}_{n}\Bigr\rvert & = \cat(n-1)-2\cat(n-2)\\
		& = \frac{1}{n}\binom{2n-2}{n-1}-\frac{2}{n-1}\binom{2n-4}{n-2}\\
		& = \left(\frac{4(2n-3)}{n(n-3)}-\frac{4}{n-3}\right)\binom{2n-5}{n-4}\\
		& = \frac{4}{n}\binom{2n-5}{n-4},
	\end{align*}
	and the claim follows from Theorem~\ref{thm:nbb_base_mobius}.
\end{proof}

Figure~\ref{fig:nbb_bases_5} illustrates the proof of Theorem~\ref{thm:pf_element_mobius} for $n=5$.  It displays the noncrossing trees corresponding to the elements of $\BB_{5}$.  We have crossed out the trees corresponding to elements of $\SM_{5}^{(2)}$ in red, to elements of $\SM_{5}^{(1)}$ in blue, and to elements of $\RR_{5}$ in green.

\begin{figure}
	\begin{displaymath}\begin{aligned}
		& \begin{tikzpicture}[scale=.5]\tiny
			\draw(2,4) node(n1){1};
			\draw(1,3) node(n2){2};
			\draw(3,3) node(n3){5};
			\draw(1,2) node(n4){3};
			\draw(1,1) node(n5){4};
			\draw(n1) -- (n2);
			\draw(n1) -- (n3);
			\draw(n2) -- (n4);
			\draw(n4) -- (n5);
			\begin{pgfonlayer}{background}
				\draw[line width=.5cm,opacity=.5,green!50!white](1.25,3.75) -- (3.75,1.25);
				\draw[line width=.5cm,opacity=.5,green!50!white](1.25,1.25) -- (3.75,3.75);
			\end{pgfonlayer}
		\end{tikzpicture}
		&& \begin{tikzpicture}[scale=.5]\tiny
			\draw(3,3) node(n1){1};
			\draw(2,2) node(n2){2};
			\draw(4,2) node(n3){5};
			\draw(1,1) node(n4){3};
			\draw(3,1) node(n5){4};
			\draw(n1) -- (n2);
			\draw(n1) -- (n3);
			\draw(n2) -- (n4);
			\draw(n2) -- (n5);
			\begin{pgfonlayer}{background}
				\draw[line width=.5cm,opacity=.5,green!50!white](1.25,3.75) -- (3.75,1.25);
				\draw[line width=.5cm,opacity=.5,green!50!white](1.25,1.25) -- (3.75,3.75);
			\end{pgfonlayer}
		\end{tikzpicture} 
		&& \begin{tikzpicture}[scale=.5]\tiny
			\draw(2,3) node(n1){1};
			\draw(1,2) node(n2){2};
			\draw(2,2) node(n3){3};
			\draw(3,2) node(n4){5};
			\draw(2,1) node(n5){4};
			\draw(n1) -- (n2);
			\draw(n1) -- (n3);
			\draw(n1) -- (n4);
			\draw(n3) -- (n5);
			\begin{pgfonlayer}{background}
				\draw[line width=.5cm,opacity=.5,green!50!white](1.25,3.75) -- (3.75,1.25);
				\draw[line width=.5cm,opacity=.5,green!50!white](1.25,1.25) -- (3.75,3.75);
			\end{pgfonlayer}
		\end{tikzpicture}
		&& \begin{tikzpicture}[scale=.5]\tiny
			\draw(2,3) node(n1){1};
			\draw(1,2) node(n2){2};
			\draw(2,2) node(n3){4};
			\draw(3,2) node(n4){5};
			\draw(1,1) node(n5){3};
			\draw(n1) -- (n2);
			\draw(n1) -- (n3);
			\draw(n1) -- (n4);
			\draw(n2) -- (n5);
			\begin{pgfonlayer}{background}
				\draw[line width=.5cm,opacity=.5,green!50!white](1.25,3.75) -- (3.75,1.25);
				\draw[line width=.5cm,opacity=.5,blue!50!white](1.25,1.25) -- (3.75,3.75);
			\end{pgfonlayer}
		\end{tikzpicture}
		&& \begin{tikzpicture}[scale=.5]\tiny
			\draw(1,3) node(n1){1};
			\draw(3,3) node(n2){2};
			\draw(2,2) node(n3){5};
			\draw(4,2) node(n4){3};
			\draw(4,1) node(n5){4};
			\draw(n1) -- (n3);
			\draw(n2) -- (n3);
			\draw(n2) -- (n4);
			\draw(n4) -- (n5);
		\end{tikzpicture}
		&& \begin{tikzpicture}[scale=.5]\tiny
			\draw(2,3) node(n1){1};
			\draw(4,3) node(n2){4};
			\draw(1,2) node(n3){2};
			\draw(3,2) node(n4){5};
			\draw(1,1) node(n5){3};
			\draw(n1) -- (n3);
			\draw(n1) -- (n4);
			\draw(n2) -- (n4);
			\draw(n3) -- (n5);
			\begin{pgfonlayer}{background}
				\draw[line width=.5cm,opacity=.5,red!50!white](1.25,3.75) -- (3.75,1.25);
				\draw[line width=.5cm,opacity=.5,red!50!white](1.25,1.25) -- (3.75,3.75);
			\end{pgfonlayer}
		\end{tikzpicture}\\
		& \begin{tikzpicture}[scale=.5]\tiny
			\draw(2,2) node(n1){1};
			\draw(4,2) node(n2){3};
			\draw(1,1) node(n3){2};
			\draw(3,1) node(n4){5};
			\draw(5,1) node(n5){4};
			\draw(n1) -- (n3);
			\draw(n1) -- (n4);
			\draw(n2) -- (n4);
			\draw(n2) -- (n5);
		\end{tikzpicture}
		&& \begin{tikzpicture}[scale=.5]\tiny
			\draw(2,2) node(n1){1};
			\draw(3,2) node(n2){3};
			\draw(4,2) node(n3){4};
			\draw(1,1) node(n4){2};
			\draw(3,1) node(n5){5};
			\draw(n1) -- (n4);
			\draw(n1) -- (n5);
			\draw(n2) -- (n5);
			\draw(n3) -- (n5);
			\begin{pgfonlayer}{background}
				\draw[line width=.5cm,opacity=.5,red!50!white](1.25,3.75) -- (3.75,1.25);
				\draw[line width=.5cm,opacity=.5,red!50!white](1.25,1.25) -- (3.75,3.75);
			\end{pgfonlayer}
		\end{tikzpicture}
		&& \begin{tikzpicture}[scale=.5]\tiny
			\draw(2,2) node(n1){1};
			\draw(4,2) node(n2){4};
			\draw(1,1) node(n3){2};
			\draw(2,1) node(n4){3};
			\draw(3,1) node(n5){5};
			\draw(n1) -- (n3);
			\draw(n1) -- (n4);
			\draw(n1) -- (n5);
			\draw(n2) -- (n5);
			\begin{pgfonlayer}{background}
				\draw[line width=.5cm,opacity=.5,red!50!white](1.25,3.75) -- (3.75,1.25);
				\draw[line width=.5cm,opacity=.5,red!50!white](1.25,1.25) -- (3.75,3.75);
			\end{pgfonlayer}
		\end{tikzpicture} 
		&& \begin{tikzpicture}[scale=.5]\tiny
			\draw(2.5,2) node(n1){1};
			\draw(1,1) node(n2){2};
			\draw(2,1) node(n3){3};
			\draw(3,1) node(n4){4};
			\draw(4,1) node(n5){5};
			\draw(n1) -- (n2);
			\draw(n1) -- (n3);
			\draw(n1) -- (n4);
			\draw(n1) -- (n5);
			\begin{pgfonlayer}{background}
				\draw[line width=.5cm,opacity=.5,green!50!white](1.25,3.75) -- (3.75,1.25);
				\draw[line width=.5cm,opacity=.5,blue!50!white](1.25,1.25) -- (3.75,3.75);
			\end{pgfonlayer}
		\end{tikzpicture}
		&& \begin{tikzpicture}[scale=.5]\tiny
			\draw(1,2) node(n1){1};
			\draw(2,2) node(n2){4};
			\draw(3,2) node(n3){2};
			\draw(2,1) node(n4){5};
			\draw(4,1) node(n5){3};
			\draw(n1) -- (n4);
			\draw(n2) -- (n4);
			\draw(n3) -- (n4);
			\draw(n3) -- (n5);
			\begin{pgfonlayer}{background}
				\draw[line width=.5cm,opacity=.5,red!50!white](1.25,3.75) -- (3.75,1.25);
				\draw[line width=.5cm,opacity=.5,red!50!white](1.25,1.25) -- (3.75,3.75);
			\end{pgfonlayer}
		\end{tikzpicture}
		&& \begin{tikzpicture}[scale=.5]\tiny
			\draw(1,2) node(n1){1};
			\draw(3,2) node(n2){2};
			\draw(2,1) node(n3){5};
			\draw(3,1) node(n4){3};
			\draw(4,1) node(n5){4};
			\draw(n1) -- (n3);
			\draw(n2) -- (n3);
			\draw(n2) -- (n4);
			\draw(n2) -- (n5);
		\end{tikzpicture}\\
		& \begin{tikzpicture}[scale=.5]\tiny
			\draw(1,2) node(n1){1};
			\draw(2,2) node(n2){2};
			\draw(3,2) node(n3){3};
			\draw(2,1) node(n4){5};
			\draw(4,1) node(n5){4};
			\draw(n1) -- (n4);
			\draw(n2) -- (n4);
			\draw(n3) -- (n4);
			\draw(n3) -- (n5);
		\end{tikzpicture}
		&& \begin{tikzpicture}[scale=.5]\tiny
			\draw(1,2) node(n1){1};
			\draw(2,2) node(n2){2};
			\draw(3,2) node(n3){3};
			\draw(4,2) node(n4){4};
			\draw(2.5,1) node(n5){5};
			\draw(n1) -- (n5);
			\draw(n2) -- (n5);
			\draw(n3) -- (n5);
			\draw(n4) -- (n5);
			\begin{pgfonlayer}{background}
				\draw[line width=.5cm,opacity=.5,red!50!white](1.25,3.75) -- (3.75,1.25);
				\draw[line width=.5cm,opacity=.5,red!50!white](1.25,1.25) -- (3.75,3.75);
			\end{pgfonlayer}
		\end{tikzpicture}
	\end{aligned}\end{displaymath}
	\caption{The noncrossing trees corresponding to the $\nc$-NBB bases for $\one$ in $(\nc_{5},\dref)$.  We have crossed out certain trees as indicated in the proof of Theorem~\ref{thm:pf_element_mobius}.}
	\label{fig:nbb_bases_5}
\end{figure}
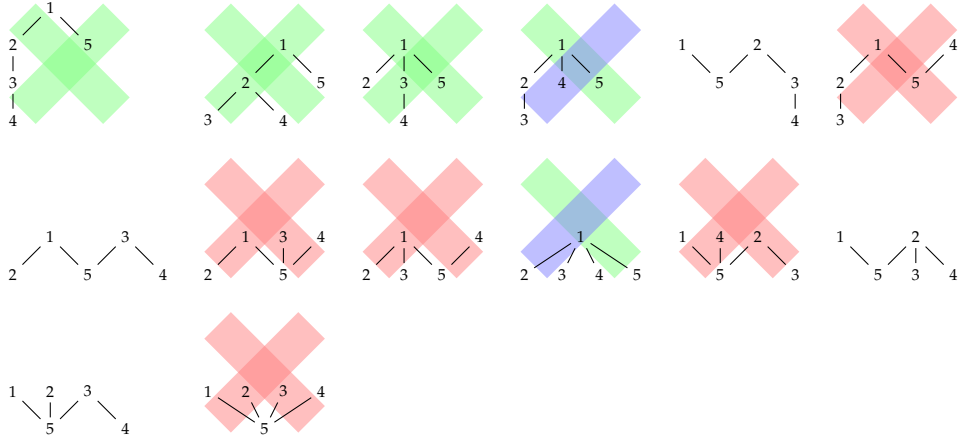

We can use the combinatorial model from above to compute $\nc$-NBB bases for any element of $\nc_{n}$, by simply picking at most one element of each rank of $(\AA_{n},\trianglelefteq)$ keeping the restriction that their join in the partition lattice is again noncrossing.  This process works since every interval in $(\nc_{n},\dref)$ is a direct product of smaller noncrossing partition lattices.  The analogous procedure for $(\pe_{n},\dref)$ does not work, due to the extra condition for $\pe$-NBB bases (Lemma~\ref{lem:atoms_bb_3}).  Moreover, the subintervals of $(\pe_{n},\dref)$ do not factor nicely into direct products of smaller lattices.  Consider the interval $[\aa_{n-2,n-1},\one]$ in $(\pe_{n},\dref)$.  The cardinalities of these intervals for $n\in\{4,5,\ldots,9\}$ are $4,12,37,118,387,1298$, and we observe that large prime factors appear in this sequence.  It seems, however, that every proper interval of $(\pe_{n},\dref)$ can be written as a direct product of an interval of the previous form and some noncrossing partition lattice. 


\section{A Subposet of $(\pe_{n},\dref)$}
	\label{sec:pf_chain_poset}
Now we consider a subposet of $(\pe_{n},\dref)$ that was introduced in \cite{bruce16decomposition}.  To that end recall that a function $f:[n]\to[n]$ is a \defn{parking function} if for all $k\in[n]$ the cardinality of $f^{-1}\bigl([k]\bigr)$ is at least $k$.  It is a classical result that the number of parking functions of length $n$ is $(n+1)^{n-1}$~\cite{haiman94conjectures}*{Proposition~2.6.1}.  

For two noncrossing partitions $\xx$ and $\yy$ with $\xx\lessdot_{\text{\normalfont dref}}\yy$, there are two unique blocks $B_{1}$ and $B_{2}$ of $\xx$ such that $B_{1}\cup B_{2}$ is a block of $\yy$.  Suppose without loss of generality that $\min B_{1}<\min B_{2}$, and define 
\begin{equation}\label{eq:parking_label}
	\pi(\xx,\yy)=\max\{j\in B_{1}\mid j\leq i\;\text{for all}\;i\in B_{2}\}.
\end{equation}
Clearly $\pi$ extends to an edge-labeling of $(\nc_{n},\dref)$; the \defn{parking labeling}.  Let $\CC_{n}$ denote the set of maximal chains of $(\nc_{n},\dref)$.  For any $X\in\CC_{n}$ the sequence $\pi(X)$ is a parking function of length $n-1$, and every such parking function arises in this way \cite{stanley97parking}*{Theorem~3.1}.  As a consequence $\bigl\lvert\CC_{n}\bigr\rvert=n^{n-2}$.  

Now let $\DD_{n}=\bigl\{X\in\CC_{n}\mid n-1\notin\pi(X)\bigr\}$ be the set of all maximal chains of $(\nc_{n},\dref)$ whose parking labeling does not contain the value $n-1$.  Let $\LL_{n}$ be the subposet of $(\nc_{n},\dref)$ whose maximal chains are precisely $\DD_{n}$, see \cite{bruce16decomposition}*{Definition~3.3}.  

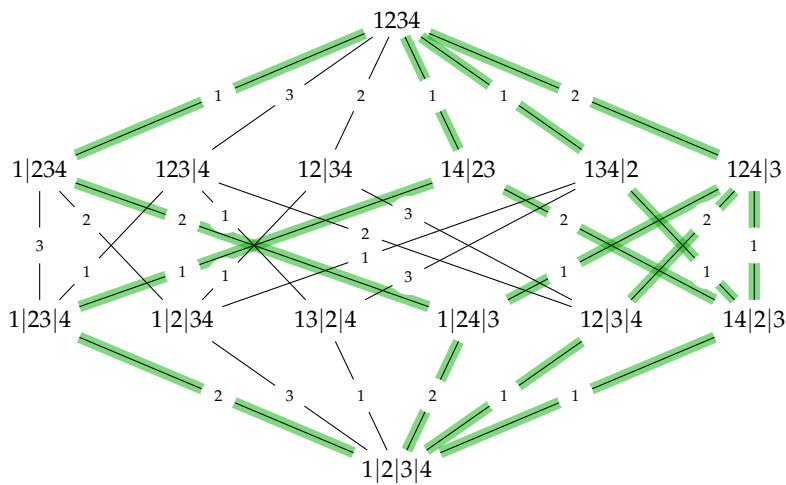
\begin{figure}
	\centering
	\begin{tikzpicture}\small
		\def\x{1.9};
		\def\y{2};
		\draw(3.5*\x,1*\y) node(n1){$1|2|3|4$};
		\draw(1*\x,2*\y) node(n2){$1|23|4$};
		\draw(2*\x,2*\y) node(n3){$1|2|34$};
		\draw(3*\x,2*\y) node(n4){$13|2|4$};
		\draw(4*\x,2*\y) node(n5){$1|24|3$};
		\draw(5*\x,2*\y) node(n6){$12|3|4$};
		\draw(6*\x,2*\y) node(n7){$14|2|3$};
		\draw(1*\x,3*\y) node(n8){$1|234$};
		\draw(2*\x,3*\y) node(n9){$123|4$};
		\draw(3*\x,3*\y) node(n10){$12|34$};
		\draw(4*\x,3*\y) node(n12){$14|23$};
		\draw(5*\x,3*\y) node(n13){$134|2$};
		\draw(6*\x,3*\y) node(n14){$124|3$};
		\draw(3.5*\x,4*\y) node(n15){$1234$};
		\draw(n1) -- (n2) node[fill=white,circle] at (2.25*\x,1.5*\y){\tiny $2$};
		\draw(n1) -- (n3) node[fill=white,circle] at (2.75*\x,1.5*\y){\tiny $3$};
		\draw(n1) -- (n4) node[fill=white,circle] at (3.25*\x,1.5*\y){\tiny $1$};
		\draw(n1) -- (n5) node[fill=white,circle] at (3.75*\x,1.5*\y){\tiny $2$};
		\draw(n1) -- (n6) node[fill=white,circle] at (4.25*\x,1.5*\y){\tiny $1$};
		\draw(n1) -- (n7) node[fill=white,circle] at (4.75*\x,1.5*\y){\tiny $1$};
		\draw(n2) -- (n8) node[fill=white,circle] at (1*\x,2.5*\y){\tiny $3$};
		\draw(n2) -- (n9) node[fill=white,circle] at (1.33*\x,2.33*\y){\tiny $1$};
		\draw(n2) -- (n12) node[fill=white,circle] at (2*\x,2.33*\y){\tiny $1$};
		\draw(n3) -- (n8) node[fill=white,circle] at (1.33*\x,2.67*\y){\tiny $2$};
		\draw(n3) -- (n10) node[fill=white,circle] at (2.3*\x,2.3*\y){\tiny $1$};
		\draw(n3) -- (n13) node[fill=white,circle] at (3.28*\x,2.42*\y){\tiny $1$};
		\draw(n4) -- (n9) node[fill=white,circle] at (2.3*\x,2.7*\y){\tiny $1$};
		\draw(n4) -- (n13) node[fill=white,circle] at (3.58*\x,2.29*\y){\tiny $3$};
		\draw(n5) -- (n8) node[fill=white,circle] at (2*\x,2.67*\y){\tiny $2$};
		\draw(n5) -- (n14) node[fill=white,circle] at (4.67*\x,2.33*\y){\tiny $1$};
		\draw(n6) -- (n9) node[fill=white,circle] at (3.28*\x,2.58*\y){\tiny $2$};
		\draw(n6) -- (n10) node[fill=white,circle] at (3.58*\x,2.71*\y){\tiny $3$};
		\draw(n6) -- (n14) node[fill=white,circle] at (5.67*\x,2.67*\y){\tiny $2$};
		\draw(n7) -- (n12) node[fill=white,circle] at (4.67*\x,2.67*\y){\tiny $2$};
		\draw(n7) -- (n13) node[fill=white,circle] at (5.67*\x,2.33*\y){\tiny $1$};
		\draw(n7) -- (n14) node[fill=white,circle] at (6*\x,2.5*\y){\tiny $1$};
		\draw(n8) -- (n15) node[fill=white,circle] at (2.25*\x,3.5*\y){\tiny $1$};
		\draw(n9) -- (n15) node[fill=white,circle] at (2.75*\x,3.5*\y){\tiny $3$};
		\draw(n10) -- (n15) node[fill=white,circle] at (3.25*\x,3.5*\y){\tiny $2$};
		\draw(n12) -- (n15) node[fill=white,circle] at (3.75*\x,3.5*\y){\tiny $1$};
		\draw(n13) -- (n15) node[fill=white,circle] at (4.25*\x,3.5*\y){\tiny $1$};
		\draw(n14) -- (n15) node[fill=white,circle] at (4.75*\x,3.5*\y){\tiny $2$};
		\begin{pgfonlayer}{background}
			\draw[opacity=.5,line width=.15cm,green!70!black](n1) -- (n2) -- (n12) -- (n15);
			\draw[opacity=.5,line width=.15cm,green!70!black](n1) -- (n5) -- (n8) -- (n15);
			\draw[opacity=.5,line width=.15cm,green!70!black](n5) -- (n14);
			\draw[opacity=.5,line width=.15cm,green!70!black](n7) -- (n12);
			\draw[opacity=.5,line width=.15cm,green!70!black](n1) -- (n6) -- (n14) -- (n15);
			\draw[opacity=.5,line width=.15cm,green!70!black](n1) -- (n7) -- (n13) -- (n15);
			\draw[opacity=.5,line width=.15cm,green!70!black](n7) -- (n14);
		\end{pgfonlayer}
	\end{tikzpicture}
	\caption{The lattice $(\nc_{4},\dref)$ with its parking labeling.  The highlighted chains form $\DD_{4}$.}
	\label{fig:partitions_4_parking}
\end{figure}

\begin{proposition}[\cite{bruce16decomposition}*{Proposition~3.4}]
	For $n\geq 3$, the ground set of $\LL_{n}$ is precisely $\pe_{n}$.
\end{proposition}

\begin{figure}
	\centering
	\begin{tikzpicture}\small
		\def\x{2};
		\def\y{2};
		\draw(2.5*\x,1*\y) node(n1){$1|2|3|4$};
		\draw(1*\x,2*\y) node(n2){$1|24|3$};
		\draw(2*\x,2*\y) node(n3){$1|23|4$};
		\draw(3*\x,2*\y) node(n4){$12|3|4$};
		\draw(4*\x,2*\y) node(n5){$14|2|3$};
		\draw(1*\x,3*\y) node(n6){$1|234$};
		\draw(2*\x,3*\y) node(n7){$124|3$};
		\draw(3*\x,3*\y) node(n8){$14|23$};
		\draw(4*\x,3*\y) node(n9){$134|2$};
		\draw(2.5*\x,4*\y) node(n10){$1234$};
		\draw(n1) -- (n2) node[fill=white] at (1.75*\x,1.5*\y){\tiny $2$};
		\draw(n1) -- (n3) node[fill=white] at (2.25*\x,1.5*\y){\tiny $3$};
		\draw(n1) -- (n4) node[fill=white] at (2.75*\x,1.5*\y){\tiny $2$};
		\draw(n1) -- (n5) node[fill=white] at (3.25*\x,1.5*\y){\tiny $1$};
		\draw(n2) -- (n6) node[fill=white] at (1*\x,2.5*\y){\tiny $3$};
		\draw(n2) -- (n7) node[fill=white] at (1.33*\x,2.33*\y){\tiny $1$};
		\draw(n3) -- (n8) node[fill=white] at (2.33*\x,2.33*\y){\tiny $1$};
		\draw(n4) -- (n7) node[fill=white] at (2.67*\x,2.33*\y){\tiny $1$};
		\draw(n5) -- (n7) node[fill=white] at (3*\x,2.5*\y){\tiny $2$};
		\draw(n5) -- (n8) node[fill=white] at (3.5*\x,2.5*\y){\tiny $3$};
		\draw(n5) -- (n9) node[fill=white] at (4*\x,2.5*\y){\tiny $3$};
		\draw(n6) -- (n10) node[fill=white] at (1.75*\x,3.5*\y){\tiny $1$};
		\draw(n7) -- (n10) node[fill=white] at (2.25*\x,3.5*\y){\tiny $3$};
		\draw(n8) -- (n10) node[fill=white] at (2.75*\x,3.5*\y){\tiny $2$};
		\draw(n9) -- (n10) node[fill=white] at (3.25*\x,3.5*\y){\tiny $2$};
		\begin{pgfonlayer}{background}
			\draw[opacity=.5,line width=.15cm,green!70!black](n1) -- (n5) -- (n7) -- (n10);
		\end{pgfonlayer}
	\end{tikzpicture}
	\caption{The poset $(\pe_{4},\pchn)$.  The labeling is inherited from $(\pe_{4},\dref)$, see Figure~\ref{fig:pf_element_4}.}
	\label{fig:pf_chain_4}
\end{figure}
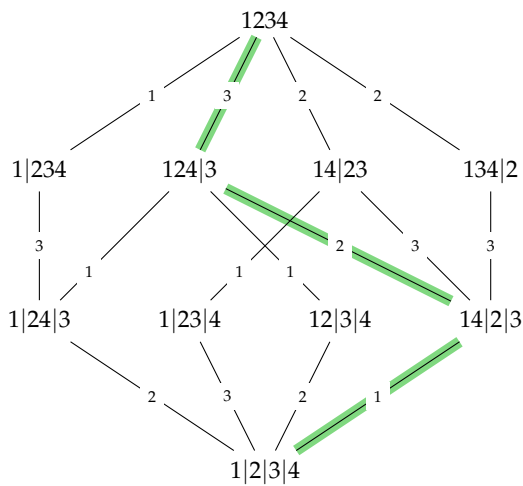

We can therefore write $\LL_{n}=(\pe_{n},\pchn)$, where $\pchn$ is a subset of $\dref$.  Figure~\ref{fig:pf_chain_4} shows the poset $(\pe_{4},\pchn)$.  This poset was extensively studied in \cite{bruce16decomposition}.  For our purposes the next statement is the most relevant.

\begin{theorem}[\cite{bruce16decomposition}*{Theorem~C}]\label{thm:pf_chain_mobius}
	For $n\geq 3$ we have $\mu_{(\pe_{n},\pchn)}(\zero,\one)=0$.
\end{theorem}

The main goal of this section is to prove Theorem~\ref{thm:pf_chain_shellable} and Corollary~\ref{cor:pf_chain_contractible}, which essentially proves the conjecture in \cite{bruce16decomposition}.  To that end we show that the restriction of the EL-labeling of $(\pe_{n},\dref)$ coming from the left-modular chain \eqref{eq:good_chain} is an EL-labeling of $(\pe_{n},\pchn)$.  First we need to show that the property of being an EL-labeling is preserved under removing particular cover relations.

\begin{proposition}\label{prop:removing_edge_shellable}
	Let $\LL=(L,\leq)$ be a bounded graded poset with an EL-labeling $\lambda$.  Let $x,y\in L\setminus\{\hat{0},\hat{1}\}$ with $x\lessdot y$.  Let $\LL'$ be the poset that arises from $\LL$ by removing the cover relation $(x,y)$.  If there is some $y'\in L$ with $x\lessdot y'$ and $\lambda(x,y)\succ\lambda(x,y')$, then the restriction of $\lambda$ to $\LL'$ is again an EL-labeling.
\end{proposition}
\begin{proof}
	Let $\lambda'$ denote the restriction of $\lambda$ to $\LL'$, and let $x,y$ be the elements from the statement.  We proceed by contraposition and suppose that $\lambda'$ is not an EL-labeling of $\LL'$.  
	
	Note that $\CC(\LL')\subseteq\CC(\LL)$, and for $X\in\CC(\LL')$ we have $\lambda'(X)=\lambda(X)$.  Since $\lambda'$ is not an EL-labeling of $\LL'$, there must be some interval $I'$ in $\LL'$ in which the EL-property of $\lambda'$ fails.  We conclude that $x,y\in I'$ (since otherwise $\lambda\equiv\lambda'$ on $I'$, which is a contradiction).  We can moreover assume without loss of generality that $x$ is the least element of $I'$, \ie $I'=[x,z]$ for some $z$.  Let $I$ be the corresponding interval in $\LL$.  There are three possibilities for $\lambda'$ to fail to be an EL-labeling of $I'$.  The existence of more than one rising maximal chain in $I'$ contradicts the assumption that $\lambda$ is an EL-labeling of $I$, and the same holds for the assumption that the unique rising chain of $I'$ is not lexicographically first.  It follows that there does not exist a rising maximal chain in $I'$.  Since there is a rising maximal chain $X$ in $I$, we conclude that $x,y\in X$; in particular $x$ is the first and $y$ is the second element of $X$.  Since $\lambda$ is an EL-labeling of $\LL$, we conclude that $\lambda(x,y)\preceq\lambda(x,y')$ for all $y'\in L$ with $x\lessdot y'$.
\end{proof}

By definition $(\pe_{n},\pchn)$ is obtained from $(\pe_{n},\dref)$ by removing certain cover relations, and the next results states that these satisfy the condition from Proposition~\ref{prop:removing_edge_shellable}.

\begin{proposition}\label{prop:pf_chain_removed_edge_label}
	Let $\xx,\yy\in\pe_{n}$ such that $\pi(\xx,\yy)=n-1$, where $\pi$ is the labeling defined in \eqref{eq:parking_label}.  There exists $\yy'\in\pe_{n}$ with $\xx\lessdot_{\text{\normalfont dref}}\yy'$ such that $\pi(\xx,\yy')<n-1$ and $\lambda(\xx,\yy)>\lambda(\xx,\yy')$, where $\lambda$ is the EL-labeling of $(\pe_{n},\dref)$ coming from the left-modular chain \eqref{eq:good_chain}.
\end{proposition}
\begin{proof}
	Let $\xx$ and $\yy$ be as desired.  Since $\pi(\xx,\yy)=n-1$, there must be a block $B$ of $\xx$ containing $n-1$, and $\{n\}$ must be a singleton block of $\xx$.  Moreover, $\yy$ must contain the block $B\cup\{n\}$.  Since $\xx\in\pe_{n}$ we conclude that $B\neq\{n-1\}$ and $1\notin B$; in particular $\xx\neq\zero$ and $\yy\neq\one$.  Let $A$ be the block of $\xx$ containing $1$.  Let $\yy'$ be the partition that contains all blocks of $\xx$ except that $A$ and $\{n\}$ are replaced by $A\cup\{n\}$.  Since $\xx\in\pe_{n}$, the blocks $A$ and $B$ cannot be crossing, which implies that $\yy'\in\pe_{n}$.  Moreover, we have $\xx\lessdot_{\text{\normalfont dref}}\yy'$.  We claim that $\yy'$ is the desired element.  
	
	First of all $\pi(\xx,\yy')<n-1$, since $n-1\notin A$, so that the cover relation $\xx\lessdot_{\text{\normalfont dref}}\yy'$ is still present in $(\pe_{n},\pchn)$.  
	
	Recall that the left-modular chain \eqref{eq:good_chain} of $(\pe_{n},\dref)$ consists of the elements $\xx_{i}$ given by the unique non-singleton block $[i-1]\cup\{n\}$.  Since $(\pe_{n},\dref)$ is supersolvable (Theorem~\ref{thm:pf_element_supersolvable}), it follows from the results in \cite{liu99left} (see also \cite{thomas06analogue}*{Proposition~2}) that the labeling $\lambda$ defined in \eqref{eq:left_modular_labeling} is equivalent to the labeling
	\begin{displaymath}
		\lambda(\ww,\zz) = \min\{i-1\mid\xx_{i}\not\dref\ww\;\text{and}\;\xx_{i}\dref\zz\}.
	\end{displaymath}
	(The ``$-1$'' in this definition comes from the fact that we label the elements in \eqref{eq:good_chain} by $\xx_{1},\xx_{2},\ldots,\xx_{n}$, and we want a labeling using the label set $[n-1]$.)  
	
	Observe that $\xx_{2}\dref\yy'$, since $\{1,n\}$ is the unique non-singleton block of $\xx_{2}$, and $1\sim_{\yy'}n$.   Since $1\not\sim_{\xx}n$, we conclude $\xx_{2}\not\dref\xx$, which implies $\lambda(\xx,\yy')=1$.  On the other hand, $1\not\sim_{\yy}n$, which implies $\xx_{2}\not\dref\yy$.  We thus have $\lambda(\xx,\yy)>1=\lambda(\xx,\yy')$.  (In fact we have $\lambda(\xx,\yy)=k$, where $k=\min B$.)
\end{proof}

We conclude this article with the remaining proofs.

\begin{proof}[Proof of Theorem~\ref{thm:pf_chain_shellable}]
	This follows by construction from Propositions~\ref{prop:removing_edge_shellable} and \ref{prop:pf_chain_removed_edge_label}.
\end{proof}

\begin{proof}[Proof of Corollary~\ref{cor:pf_chain_contractible}]
	This follows from Theorem~\ref{thm:pf_chain_shellable} and Theorems~\ref{thm:pf_chain_mobius} and \ref{thm:shellable_wedge}.
\end{proof}

\bibliography{../../literature}

\begin{bibdiv}
\begin{biblist}

\bib{bjorner80shellable}{article}{
      author={Bj{\"o}rner, Anders},
       title={{Shellable and Cohen-Macaulay Partially Ordered Sets}},
        date={1980},
     journal={Transactions of the American Mathematical Society},
      volume={260},
       pages={159\ndash 183},
}

\bib{bjorner96shellable}{article}{
      author={Bj{\"o}rner, Anders},
      author={Wachs, Michelle~L.},
       title={{Shellable and Nonpure Complexes and Posets I}},
        date={1996},
     journal={Transactions of the American Mathematical Society},
      volume={348},
       pages={1299\ndash 1327},
}

\bib{blass97mobius}{article}{
      author={Blass, Andreas},
      author={Sagan, Bruce~E.},
       title={{M{\"o}bius Functions of Lattices}},
        date={1997},
     journal={Advances in Mathematics},
      volume={127},
       pages={94\ndash 123},
}

\bib{bruce16decomposition}{article}{
      author={Bruce, Melody},
      author={Dougherty, Michael},
      author={Hlavacek, Max},
      author={Kudo, Ryo},
      author={Nicolas, Ian},
       title={{A Decomposition of Parking Functions by Undesired Spaces}},
        date={2016},
     journal={The Electronic Journal of Combinatorics},
      volume={23},
}

\bib{haglund08catalan}{book}{
      author={Haglund, James},
       title={{The $q,t$-Catalan Numbers and the Space of Diagonal Harmonics}},
   publisher={American Mathematical Society},
     address={Providence, RI},
        date={2008},
}

\bib{haiman94conjectures}{article}{
      author={Haiman, Mark},
       title={{Conjectures on the Quotient Ring by Diagonal Invariants}},
        date={1994},
     journal={Journal of Algebraic Combinatorics},
      volume={3},
       pages={17\ndash 76},
}

\bib{hersh99decomposition}{thesis}{
      author={Hersh, Patricia},
       title={{Decomposition and Enumeration in Partially Ordered Sets}},
        type={Ph.D. Thesis},
 institution={Massachusetts Institute of Technology},
        date={1999},
}

\bib{konheim66occupancy}{article}{
      author={Konheim, Alan~G.},
      author={Weiss, Benjamin},
       title={{An Occupancy Discipline and Applications}},
        date={1966},
     journal={SIAM Journal on Applied Mathematics},
      volume={14},
       pages={1266\ndash 1274},
}

\bib{kreweras72sur}{article}{
      author={Kreweras, Germain},
       title={Sur les partitions non crois\'ees d'un cycle},
        date={1972},
     journal={Discrete Mathematics},
      volume={1},
       pages={333\ndash 350},
}

\bib{liu99left}{thesis}{
      author={Liu, Shu-Chung},
       title={{Left-Modular Elements and Edge-Labellings}},
        type={Ph.D. Thesis},
 institution={Michigan State University},
        date={1999},
}

\bib{mccammond06noncrossing}{article}{
      author={McCammond, Jon},
       title={{Noncrossing Partitions in Surprising Locations}},
        date={2006},
     journal={American Mathematical Monthly},
      volume={113},
       pages={598\ndash 610},
}

\bib{mcnamara03labelings}{article}{
      author={McNamara, Peter},
       title={{EL-Labelings, Supersolvability, and 0-Hecke Algebra Actions on
  Posets}},
        date={2003},
     journal={Journal of Combinatorial Theory (Series A)},
      volume={101},
       pages={69\ndash 89},
}

\bib{mcnamara06poset}{article}{
      author={McNamara, Peter},
      author={Thomas, Hugh},
       title={{Poset Edge-Labellings and Left Modularity}},
        date={2006},
     journal={European Journal of Combinatorics},
      volume={27},
       pages={101\ndash 113},
}

\bib{segner61enumeratio}{article}{
      author={Segner, Johann~A.},
       title={{Enumeratio Modorum quibus Figurae Planae Rectilineae per
  Diagonales Dividuntur in Triangula}},
        date={1761},
     journal={Novi Commentarii Academiae Scientiarum Imperialis
  Petropolitanae},
      volume={VII},
       pages={203\ndash 210},
}

\bib{simion00noncrossing}{article}{
      author={Simion, Rodica},
       title={{Noncrossing Partitions}},
        date={2000},
     journal={Discrete Mathematics},
      volume={217},
       pages={397\ndash 409},
}

\bib{sloane}{misc}{
      author={Sloane, Neil J.~A.},
       title={{The Online Encyclopedia of Integer Sequences}},
        note={\url{http://www.oeis.org}},
}

\bib{stanley72supersolvable}{article}{
      author={Stanley, Richard~P.},
       title={{Supersolvable Lattices}},
        date={1972},
     journal={Algebra Universalis},
      volume={2},
       pages={197\ndash 217},
}

\bib{stanley97parking}{article}{
      author={Stanley, Richard~P.},
       title={{Parking Functions and Noncrossing Partitions}},
        date={1997},
     journal={The Electronic Journal of Combinatorics},
      volume={4},
}

\bib{thomas06analogue}{article}{
      author={Thomas, Hugh},
       title={{An Analogue of Distributivity for Ungraded Lattices}},
        date={2006},
     journal={Order},
      volume={23},
       pages={249\ndash 269},
}

\end{biblist}
\end{bibdiv}

\end{document}